\documentclass[12pt]{amsart}
 
\pdfoutput=1
\usepackage[margin=1in]{geometry}
\usepackage{amsmath,amsthm,amssymb,amsrefs,bbm,color,esint,esvect,float,graphicx,mathrsfs}

\usepackage[foot]{amsaddr}

\usepackage{euscript,latexsym}
\usepackage{accents}

\newcommand{\norm}[1]{\ensuremath{\left\|#1\right\|}} 
\newcommand{\abs}[1]{\ensuremath{\left\vert#1\right\vert}}
\newcommand{\ip}[2]{\ensuremath{\left\langle#1,#2\right\rangle}} 

\newcommand{\C}{\mathbb{C}}

\newcommand{\R}{\mathbb{R}}

\newcommand{\PW}{\mathcal{P}\mathcal{W}}
\newcommand{\FF}{\mathcal{F}}
 
 \newtheorem{thm}{Theorem}[section]
 
 \newtheorem{cor}[thm]{Corollary}
 \newtheorem{prop}[thm]{Proposition}
 \newtheorem{rem}[thm]{Remark}

 \numberwithin{equation}{section}
 
\newenvironment{theorem}[2][Theorem]{\begin{trivlist}
\item[\hskip \labelsep {\bfseries #1}\hskip \labelsep {\bfseries #2.}]}{\end{trivlist}}

\theoremstyle{definition}

\begin{document}

\title[Riesz-Kolmogorov type compactness criteria with applications]{Riesz-Kolmogorov type compactness criteria in function spaces with applications}
\author{Mishko Mitkovski}
\email{mmitkov@clemson.edu}
\author{Cody B. Stockdale}
\address{
 School of Mathematical and Statistical Sciences; 
 Clemson University; 
 Clemson, SC, 29631, USA}
\email{cbstock@clemson.edu}
\author{Nathan A. Wagner}
\email{nathanawagner@wustl.edu}
\author{Brett D. Wick}
\address{
 Department of Mathematics and Statistics; 
 Washington University in St. Louis; 
 St. Louis, MO, 63130 USA}
\email{bwick@wustl.edu}

 \maketitle

\begin{abstract}
We present forms of the classical Riesz-Kolmogorov theorem for compactness that are applicable in a wide variety of settings. 
In particular, our theorems apply to classify the precompact subsets of the Lebesgue space $L^2$, Paley-Wiener spaces, weighted Bargmann-Fock spaces, and a scale of weighted Besov-Sobolev spaces of holomorphic functions that includes weighted Bergman spaces of general domains as well as the Hardy space and the Dirichlet space. We apply the compactness criteria to characterize the compact Toeplitz operators on the Bergman space, deduce the compactness of Hankel operators on the Hardy space, and obtain general umbrella theorems.

\bigskip
\noindent \textbf{Keywords.}  Compactness, framed spaces, spaces of holomorphic functions, Toeplitz operators, Hankel operators.

\smallskip
\noindent \textbf{MSC.}  Primary 46B50; Secondary 46E15, 46E30, 47B35.

\smallskip
\noindent \textbf{Statements and Declarations.}  The authors state that there are no conflicts of interest and that there are no data sets associated with this research.
\end{abstract}

\section{Introduction}

The Riesz-Kolmogorov theorem is a fundamental result in analysis that characterizes the precompact subsets of $L^p(\mathbb{R}^n)$. The statement is as follows.
\begin{theorem}{A}
Let $p \in [1,\infty)$. A set $\mathcal{F}\subseteq L^p(\mathbb{R}^n)$ is precompact if and only if 
$$
    \lim_{R\rightarrow \infty} \sup_{f \in \mathcal{F}} \int_{|x|>R}|f(x)|^p\,dx=0
$$
and
$$
    \lim_{|h|\rightarrow 0}\sup_{f \in \mathcal{F}}\int_{\mathbb{R}^n}|f(x-h)-f(x)|^p\,dx=0.
$$
\end{theorem}
\noindent Theorem A is classically presented with the additional condition of $\mathcal{F}$ being a bounded subset of $L^p(\mathbb{R}^n)$, however this condition is redundant as it is implied by the other two conditions of the theorem, see \cite{HoHM2019}.

The Riesz-Kolmogorov criterion is named after the work of Kolmogorov and Riesz from \cite{K1931} and \cite{R1933}, respectively. In \cite{K1931}, Kolmogorov proved a version of Theorem A in the case when $1<p<\infty$ and all functions in $\mathcal{F}$ are supported on a common bounded set. Riesz independently discovered a version of Theorem A in \cite{R1933} in the case $1\leq p < \infty$. See \cite{HoH2010} for a more detailed historical accounting of this topic.

The Riesz-Kolmogorov characterization has been adapted to handle many other situations. For example, Fr\'echet proved a version of the theorem that includes arbitrary $p>0$ in \cite{F1937}, Phillips characterized precompact subsets of $L^p$ with respect to arbitrary measure spaces in \cite{P1940}, Weil obtained a version of the theorem in the setting of locally compact groups in \cite{W1940}, and Takahashi proved a version of the theorem for Orlicz spaces in \cite{T1934}. There are also versions of the precompactness criterion for weighted settings in \cites{CC2013, GZ2020} and matrix weighted settings in \cite{LYZ2021}. See \cites{AU2020,B2017,CDLW2019,DLMWY2018,F1982,GW1970,GM2015,GM2014,GR2016,IK2009,K1999,M1983,R2009} for further references.

As shown in \cite{HoHM2019}*{Theorem 4} or \cite{B1987}*{p. 466} the Riesz-Kolmogorov theorem can be proved using the following more abstract compactness criterion of Mazur.

\begin{theorem}{B}
Let $\mathcal{X}$ be a Banach space and suppose that $\{T_n\}_{n=1}^{\infty}$ is a sequence of compact operators on $\mathcal{X}$ that converges to the identity in the strong operator topology; that is, $\lim_{n\rightarrow\infty}\|T_nf-f\|_{\mathcal{X}}=0$ for all $f \in \mathcal{X}$. A bounded set $\mathcal{F}\subseteq \mathcal{X}$ is precompact if and only if 
$$
    \lim_{n\rightarrow\infty}\sup_{f\in\mathcal{F}}\|T_nf -f\|_{\mathcal{X}}=0.
$$
\end{theorem}
In \cite{P1940}*{Theorem 3.7}, Phillips proved a very similar theorem and applied it to characterize the precompact subsets of $L^p$ with respect to arbitrary measure spaces. In \cite{S1957}, Sudakov showed that if at least one of the operators $T_n$ does not have $1$ as an eigenvalue, then the boundedness condition on $\mathcal{F}$ in Theorem B is not needed (see also \cite{HoHM2019}*{p. 90--91}). 

\begin{proof}[Proof of Theorem B] 
First suppose that $\mathcal{F}$ is precompact. By the uniform boundedness principle, $B:=\sup_{n\in\mathbb{N}}\|T_n\|_{\mathcal{X}\rightarrow\mathcal{X}}<\infty$. Let $\varepsilon>0$. Since $\mathcal{F}$ is precompact, there exists a finite subset $\{f_1,\ldots,f_K\}\subseteq \mathcal{F}$ such that for each $f \in \mathcal{F}$ there exists $1\leq j\leq K$ with $\|f_j-f\|_{\mathcal{X}}<\frac{\varepsilon}{3}\min(B,1)$. Choose $N$ so that $\|T_nf_j-f_j\|_{\mathcal{X}}<\frac{\varepsilon}{3}$ for all $n \ge N$ and all $1\leq j\leq K$. For $f \in \mathcal{F}$ and $n\ge N$, let $1\leq j\leq K$ be such that $\|f_j-f\|_{\mathcal{X}}<\frac{\varepsilon}{3}$ and note
$$
    \|T_nf-f\|_{\mathcal{X}} \leq \|T_n(f-f_j)\|_{\mathcal{X}}+\|T_nf_j-f_j\|_{\mathcal{X}}+\|f_j-f\|_{\mathcal{X}} < \varepsilon.
$$

Assuming the uniform strong operator topology convergence of $T_n$ to the identity, we have that for any $\varepsilon>0$ there exists $N \in \mathbb{N}$ such that $\text{dist}(f, T_N\mathcal{F})< \varepsilon$ for all $f \in \mathcal{F}$. Since $\mathcal{F}$ is bounded and $T_N$ is compact, $T_N\mathcal{F}$ is precompact. The precompactness of $\mathcal{F}$ follows.
\end{proof}

We observe that a slight strengthening of Mazur's Theorem B can be obtained in a Hilbert space setting by relaxing the norm conditions involving $\|T_nf-f\|_{\mathcal{X}}$ to quadratic form conditions. This result is likely already known, but we were unable to find a reference. 
\begin{thm}\label{HilbertRK}
Let $\mathcal{H}$ be a Hilbert space and suppose that $\{T_n\}_{n=1}^{\infty}$ is a sequence of compact operators on $\mathcal{H}$ such that $\lim_{n\rightarrow\infty}\langle T_nf-f,f\rangle_{\mathcal{H}}=0$ for all $f \in \mathcal{\mathcal{H}}$. A bounded set $\mathcal{F}\subseteq \mathcal{H}$ is precompact if and only if 
$$
    \lim_{n\rightarrow\infty}\sup_{f\in\mathcal{F}}\langle T_nf -f, f\rangle_{\mathcal{H}}=0.
$$
\end{thm}

The usual way to derive the Riesz-Kolmogorov theorem when all functions in $\mathcal{F}$ are supported on a common bounded set from Mazur's Theorem B is to use the averaging operators 
$$
    T_nf(x)=\frac{1}{V(B(x,\frac{1}{n}))}\int_{B(x,\frac{1}{n})}f(y)\,dy= f * \frac{1}{V(B(0,\frac{1}{n}))}\chi_{B(0,\frac{1}{n})}(x),
$$ 
where $V$ denotes the Lebesgue measure, see \cites{HoHM2019,S1957}. Loosely speaking, the Riesz-Kolmogorov theorem says that for a set $\mathcal{F}$ to be compact, all of its elements need to have uniformly small tails on the spatial side (first condition) and on the frequency side (second condition). Therefore, to use Mazur's theorem to derive a compactness criterion of Riesz-Kolmogorov type, one must use operators $T_n$ that ``truncate" in both of the spatial and frequency domains. The simplest application of this idea gives the following theorem. 

\begin{theorem}{C}
A bounded set $\mathcal{F}\subseteq L^2(\mathbb{R}^n)$ is precompact if and only if
$$
    \lim_{R\rightarrow \infty} \sup_{f \in \mathcal{F}} \int_{|x|>R}|f(x)|^2\,dx=0 \quad  \text{and} \quad
    \lim_{R\rightarrow \infty} \sup_{f \in \mathcal{F}} \int_{|\xi|>R}|\hat{f}(\xi)|^2\,d\xi=0.
$$
\end{theorem}

Theorem C inspired the work of D\"orfler, Feichtinger, and Gr\"ochenig in \cite{DFG2002} where they derived compactness criteria for modulation spaces and co-orbit spaces using the short-time Fourier transform. Recall that the short-time Fourier transform $S_\phi: L^2(\R^n)\to L^2(\R^{2n})$ with a window function $\phi\in L^2(\R^n)$ is defined by $S_\phi f(a,b):=\ip{f}{\phi_{(a,b)}},$ where $\phi_{(a,b)}(x):=e^{2\pi i b x}\phi(x-a)$. The most classical window $\phi$ is the Gaussian window. The following is the compactness characterization in terms of the short-time Fourier transform 
obtained 
in~\cite{DFG2002}.  
\begin{theorem}{D}
A bounded set $\mathcal{F}\subseteq L^2(\mathbb{R}^n)$ is precompact if and only if
$$
    \lim_{R\rightarrow \infty} \sup_{f \in \mathcal{F}} \int_{\R^{2n}\setminus [-R,R]^{2n}}|S_\phi f(a,b)|^2\,dadb=0.    
$$
\end{theorem}
\noindent Since that Gabor basis simultaneously respects both the spatial and the frequency behavior, only one uniform decay condition is needed in Theorem D.


Our first main result is a direct generalization of Theorem D. It turns out that one can replace the Gabor system $\{\phi_{(a,b)} : (a,b)\in \R^{2n}\}$ with any continuous Parseval frame. Recall that for a Hilbert space $\mathcal{H}$, a collection $\{k_x\}\subseteq \mathcal{H}$ indexed by a measure space $(X,\mu)$ is a continuous Parseval frame for $\mathcal{H}$ if 
$$
    \|f\|_{\mathcal{H}}^2= \int_X|\langle f,k_x\rangle_{\mathcal{H}}|^2\,d\mu(x)
$$
for each $f \in \mathcal{H}$. If $\{k_x\}_{x \in X}$ is a continuous Parseval frame for a Hilbert space $\mathcal{H}$, then  
$$
    f = \int_{X} \langle f,k_x\rangle_{\mathcal{H}} k_x\,d\mu(x)
$$
for each $f \in \mathcal{H}$. By an exhaustion for $X$ we mean a sequence of subsets of $X$, $\{F_n\}_{n=1}^{\infty}$, such that $F_n\subseteq F_{n+1}$ for each $n$ 
and $\bigcup_{n=1}^{\infty}F_n=X$.

\begin{thm}\label{HilbertFrameRK1}
Let $\mathcal{H}$ be a Hilbert space with a continuous Parseval frame $\{k_x\}$ indexed by a measure space $(X,\mu)$. Suppose that $X$ has an exhaustion $\{F_n\}_{n=1}^{\infty}$ such that $\mu(F_n)< \infty$ for all $n \in \mathbb{N}$. 
A bounded set $\mathcal{F}\subseteq \mathcal{H}$ is precompact if and only if 
$$
    \lim_{n\rightarrow\infty}\sup_{f\in\mathcal{F}}\int_{X\setminus F_n}|\langle f,k_x\rangle_{\mathcal{H}}|^2\,d\mu(x)=0.
$$
\end{thm}

Assuming more on the the frame $\{k_x\}_{x \in X}$, we may relax the finite measure assumption of Theorem \ref{HilbertFrameRK1}. The following frame-theoretic statement relies on Theorem \ref{HilbertRK}.
\begin{thm}\label{HilbertFrameRK2}
Let $\mathcal{H}$ be a Hilbert space equipped with a continuous Parseval frame $\{k_x\}$ indexed by an unbounded metric measure space $(X,d,\mu)$ satisfying for some $w:X\rightarrow (0,\infty)$
\begin{align*}
\addtolength{\itemsep}{0.2cm}
    & \displaystyle \sup_{y \in X} w(y)^{-1} \int_X |\langle k_x,k_y\rangle_{\mathcal{H}}|w(x)\,d\mu(x)<\infty,\\
    & \displaystyle \lim_{R\rightarrow \infty}\sup_{y \in X} w(y)^{-1}\int_{X\setminus B(y,R)}|\langle k_x,k_y\rangle_{\mathcal{H}}|w(x)\,d\mu(x)=0, \quad\text{and}\\
    & \displaystyle |\langle k_x,k_y\rangle_{\mathcal{H}}|\rightarrow 0 \quad\text{as}\quad d(x,y)\rightarrow \infty. 
\end{align*}
\vspace{0.2cm}
Suppose that $X$ has an exhaustion $\{F_n\}_{n=1}^{\infty}$ such that 
$$
	\lim_{d(y,y_0)\rightarrow \infty}\mu(F_n\cap B(y,R))=0
$$
for some (any) $y_0 \in X$, some (any) $R>0$, and all $n \in \mathbb{N}$.
A bounded set $\mathcal{F}\subseteq \mathcal{H}$ is precompact if and only if 
$$
    \lim_{n\rightarrow\infty}\sup_{f\in\mathcal{F}}\int_{X\setminus F_n}|\langle f,k_x\rangle_{\mathcal{H}}|^2\,d\mu(x)=0.
$$
\end{thm}

A version of Theorem \ref{HilbertFrameRK1} 
also holds in appropriate Banach space settings. For a Banach space $\mathcal{X}$, $p \in [1,\infty)$, and a measure space $(X,\mu)$, we say $(\{f_x\}_{x \in X},\{f_x^*\}_{x\in X})\subseteq \mathcal{X}\times \mathcal{X}^*$ is a continuous frame for $\mathcal{X}$ with respect to $L^p(X,\mu)$ if 
\begin{enumerate}
\addtolength{\itemsep}{0.2cm}
    \item $\displaystyle\sup_{x \in X}\|f_x^*\|_{\mathcal{X}\rightarrow \mathbb{C}}<\infty$,
    \item the function $x \mapsto \langle f,f_x^*\rangle$ is in $L^p(X,\mu)$ for all $f \in \mathcal{X}$,
    \item there exist $c,C>0$ such that 
    $$
        c\|f\|_{\mathcal{X}}\leq \|\langle f,f_x^*\rangle\|_{L^p(X,\mu)}\leq C\|f\|_{\mathcal{X}}
    $$
    for all $f \in \mathcal{X}$, and 
    \item each $f \in \mathcal{X}$ satisfies
    $$f = \int_X\langle f,f_x^*\rangle f_x\,d\mu(x).
    $$
\end{enumerate}
Note that, unlike in the Hilbert space setting, the existence of $f_x\in \mathcal{X}$ such that (4) holds is not guaranteed from  condition (3) in general Banach spaces, so their existence is assumed. 



\begin{thm}\label{BanachFrameRK1}
Let $p \in [1,\infty)$ and $\mathcal{X}$ be a reflexive Banach space equipped with a continuous frame $(\{f_x\},\{f_x^*\})$ 
with respect to $L^p(X,\mu)$. Suppose that $X$ has an exhaustion $\{F_n\}_{n=1}^{\infty}$ such that $\mu(F_n)<\infty$ for all $n \in \mathbb{N}$. A bounded set $\mathcal{F}\subseteq \mathcal{X}$ is precompact if and only if 
$$
    \lim_{n\rightarrow\infty}\sup_{f\in\mathcal{F}}\int_{X\setminus F_n}|\langle f,f_x^*\rangle|^p\,d\mu(x)=0.
$$
\end{thm}

We next extend our compactness criterion to function spaces which are not necessarily framed 
spaces. More precisely, we consider Banach function spaces consisting of functions defined on a metric measure space $(X,d,\mu)$ with a Radon measure $\mu$. 

\begin{thm}\label{BanachMeanValueRK}
Let $\mathcal{X}$ be a Banach space of functions on a metric measure space $(X, d,\mu)$ with a compact exhaustion $\{F_n\}_{n=1}^{\infty}$. 
Let $p \in [1,\infty)$ and suppose that there is a point $x_0 \in X$ and linear maps $D_j: \mathcal{X}\rightarrow C(X)$, $j=1,2,\ldots,N+M$, such that
$$
	\|f\|_{\mathcal{X}}^p=\sum_{j=1}^{N}\int_{X}|D_jf(x)|^p\,d\mu(x)+\sum_{j=N+1}^{N+M}|D_jf(x_0)|^p
$$
for all $f \in \mathcal{X}$. Suppose also that 
\begin{itemize}
\item[(i)] $\mathcal{F}\subseteq\mathcal{X}$ is bounded,
\item[(ii)] for each set $F_n$ and $1\leq j\leq N$, the collection of functions $\{D_j f : f\in \mathcal{F}\}$ is equicontinuous on $F_n$, and 
\item[(iii)] for each $x\in X$ and $1\leq j\leq N+M$, $\sup_{f \in \mathcal{F}} |D_jf(x)|<\infty$.
\end{itemize}
Then $\mathcal{F}$ is precompact if and only if 
$$
	\lim_{n\rightarrow \infty}\sup_{f\in\mathcal{F}}\sum_{j=1}^N\int_{X\setminus F_n}|D_jf(x)|^p\,d\mu(x) =0.
$$
\end{thm}

Note that Theorem \ref{BanachMeanValueRK} generalizes our Theorem \ref{HilbertFrameRK1} in the case when $x \mapsto k_x$ is continuous by taking $N=1$, $M=0$, and $Df(x)=\langle f,k_x\rangle_{\mathcal{H}}$.


\subsection{Compactness criteria in function spaces}
We now show how our results can be used to establish compactness criteria in various function spaces including the Lebesgue space $L^2(\mathbb{R}^n)$, Paley-Wiener spaces, weighted Bargmann-Fock spaces, and a scale of weighted Besov-Sobolev spaces that includes weighted Bergman spaces, the Hardy space, and the Dirichlet space. This list of applications is certainly not exhaustive -- we only mention a focused selection of well-known examples in which our results apply.

\subsubsection{The Lebesgue space $L^2(\mathbb{R}^n)$} We already presented several alternative compactness characterizations in $L^2(\mathbb{R}^n)$ besides the classical Riesz-Kolmogorov theorem. Our Theorem~\ref{HilbertFrameRK1} shows that every continuous Parseval frame provides a new compactness criterion. For example, if we use the continuous Parseval frame of wavelets indexed as usual by the $ax+b$ group $\mathbb{R}^{n+1}_+:=(0,\infty)\times\mathbb{R}^n$ equipped with the usual hyperbolic measure and metric, we obtain a compactness characterization in terms of the continuous wavelet transform. Namely, a bounded set $\mathcal{F}\subseteq L^2(\R^n)$ is compact if and only if the continuous wavelet transforms of all the elements of $\mathcal{F}$ have uniformly null tails. This fact seems to have been first noticed in~\cite{DFG2002}*{Theorem 3}. 

\subsubsection{Paley-Wiener spaces} Recall that for a Borel measurable 
set $E \subseteq \R^n$ with finite Lebesgue measure, 
the Paley-Wiener space $\PW(E)$ is the subspace of $L^2(\R^n)$ consisting of functions whose Fourier transform is supported in $E$. In the case when $E=[-a,a]^n$, 
all elements of $\PW(E)$ can be extended to entire functions with exponential type no greater than $a$. Every Paley-Wiener space is a reproducing kernel Hilbert space, and an application of the Plancherel theorem shows that the normalized reproducing kernels form a continuous Parseval frame for $\PW(E)$. Therefore our Theorem \ref{HilbertFrameRK1} immediately gives the following simple criterion for compactness in Paley-Wiener spaces.
\begin{thm} A bounded set $\mathcal{F}\subseteq \PW(E)$ is precompact if and only if 
$$
    \lim_{R\rightarrow\infty}\sup_{f\in\mathcal{F}}\int_{|x|>R}|f(x)|^2 \,dx=0.
$$
\end{thm}
\noindent We remark that in the classical case $E=[-a,a]^n$ this fact is also immediate from Theorem C since the second condition of that theorem is automatically satisfied by a family of functions in the Paley-Wiener space $\PW([-a.a]^n).$

\subsubsection{Weighted Bargmann-Fock spaces} The weighted Bargmann-Fock space $\FF_\phi(\C^n)$ is the space of all entire functions $f:\C^n \to \C$ satisfying the integrability condition 
$$\norm{f}_\phi^2 := \int_{\C^n} \abs{f(z)}^2 e^{-2 \phi(z)} \,dV(z) < \infty,$$
where 
$\phi: \C^n \to \R$ is a plurisubharmonic function such that for all $z\in \C^n$ 
\begin{equation*}\label{harmonic}
 i\partial\bar{\partial}\phi \simeq i\partial\bar{\partial}|z|^2,
\end{equation*}
in the sense of positive currents.  The classical Bargmann-Fock space $\FF(\C^n)$ is an important special case obtained when $\phi(z)=\frac{\pi}{2} |z|^2$. 
 
Equipped with the norm $\norm{\cdot}_\phi$, the weighted Bargmann-Fock space $\FF_\phi(\C^n)$ is a reproducing kernel Hilbert space. We will denote its reproducing kernel at $z$ by $K^\phi_z$. It is easy to see that the normalized reproducing kernels indexed by the metric measure space $(\C^n, V_\phi, d)$, where $dV_\phi(z) := \|K^\phi_z\|_\phi^2 e^{-2\phi(z)} dV(z)$ and $d$ is the usual Euclidean metric on $\C^n$, form a continuous Parseval frame. A straightforward application of our Theorem~\ref{HilbertFrameRK1} gives the following criterion for compactness in weighted Bargmann-Fock spaces.

\begin{thm} A bounded set $\FF\subseteq \FF_\phi(\C^n)$ is precompact if and only if 
$$
    \lim_{R\rightarrow\infty}\sup_{f\in\mathcal{F}}\int_{
    |z|>R}|f(z)|^2e^{-2\phi(z)}\, dV(z)=0.
$$
\end{thm}

\subsubsection{Weighted Besov-Sobolev spaces}

Let $D\subseteq \mathbb{C}^n$ be a bounded domain, $p \in [1,\infty)$, and $\sigma$ be an integrable weight on $D$, that is, $\sigma$ is positive almost everywhere and $\int_D\sigma\,dV<\infty$. 
In order for our spaces to be Banach spaces, we additionally suppose 
that for any compact $K \subsetneq D$, there exists a constant $C_{K,p,\sigma}>0$ such that
\begin{align}\label{BergmanInequalityCondition}
    |f(z)|^p \leq C_{K,p, \sigma} \int_{D} |f|^p \sigma \mathop{d V}
\end{align} 
for every $z \in K $ and all functions $f$ in the space. 
For $J \in \mathbb{N}$ and such $D\subseteq\mathbb{C}^n$, $p\in[1,\infty)$, and integrable weights $\sigma$, we define the weighted Besov-Sobolev space $\mathcal{B}_{\sigma}^{p,J}(D)$ to be the space of holomorphic $f:D\rightarrow \mathbb{C}$ such that \eqref{BergmanInequalityCondition} holds and
$$
    \|f\|_{\mathcal{B}_{\sigma}^{p,J}(D)}:= \left(\sum_{|\alpha|<J} \left|\frac{\partial^\alpha f}{\partial z^\alpha}(z_0)\right|^p + \sum_{|\alpha|=J}\int_{D} \left|\frac{\partial^\alpha f}{\partial z^\alpha}\right|^p \sigma \mathop{d V}\right)^{1/p}< \infty,
$$
where $\alpha$ are multi-indices indicating complex derivatives and $z_0$ is an arbitrary fixed point in $D$. For $\delta \geq 0$, define the subset $D_{\delta}:= \{z \in D: \operatorname{dist}(z,\partial D) >\delta\}$ and notice that $D_0=D.$  The following is a consequence of Theorem \ref{BanachMeanValueRK}. 

\begin{thm}\label{BesovSobolevRK}
A bounded set $\mathcal{F} \subseteq \mathcal{B}^{p,J}_{\sigma}(D)$ is precompact if and only if 
$$
    \lim_{\delta\rightarrow 0^+} \sup_{f \in \mathcal{F}} \sum_{|\alpha|=J}\int_{D \setminus D_\delta} \left|\frac{\partial^\alpha f}{\partial z^\alpha}\right|^p \sigma \mathop{d V}=0.
$$
\end{thm}

Theorem \ref{BesovSobolevRK} immediately gives compactness criteria for various function spaces including weighted Bergman spaces, the Hardy space, and the Dirichlet space.


Let $D$ be a strongly pseudoconvex domain with a $C^2$ defining function $\rho$, that is, $\rho$ is a $C^2$ plurisubharmonic function such that with $D=\{z\in \mathbb{C}^n:\rho(z)<0\}$ and $\nabla{\rho}(z) \neq 0$ for $z \in \partial D$. For $p \in [1,\infty)$ and $t>-1$, define the weighted Bergman space of $D$, $\mathcal{A}^p_t(D)$, to be the space of holomorphic $f:D\rightarrow\mathbb{C}$ such that
$$
    \|f\|_{\mathcal{A}^p_t(D)}:= \left( \int_{D} |f|^p (-\rho)^t \mathop{dV} \right)^{1/p}< \infty.
$$
Note that these spaces generalize the radially weighted Bergman spaces of the unit ball $\mathbb{B}_n\subseteq \mathbb{C}^n$ with weight $(1-|z|^2)^t$. We denote $\mathcal{A}^p(D):=\mathcal{A}^p_0(D)$.
\begin{cor}\label{RadialBergmanSpaces} 
A bounded set $\mathcal{F} \subseteq \mathcal{A}^p_t(D)$ is precompact if and only if 
$$
    \lim_{\delta\rightarrow 0^+} \sup_{f \in \mathcal{F}} \int_{D \setminus D_\delta} |f|^p (-\rho)^t \mathop{d V}=0.
$$
\end{cor}

We also apply Theorem \ref{BesovSobolevRK} to weighted Bergman spaces with respect to $B_p$ weights; see \cite{WW2020} for a definition of $B_p$ weights on $C^2$ domains. For a strongly pseudoconvex $C^2$ domain $D$, $p \in [1,\infty)$, and $\sigma \in B_p$, define the weighted Bergman space of $D$ with respect to $\sigma$, $\mathcal{A}^p_\sigma(D)$, to be the space of holomorphic $f:D\rightarrow\mathbb{C}$ such that
$$
    \|f\|_{\mathcal{A}^p_\sigma(D)}:= \left( \int_{D} |f|^p \sigma \mathop{dV} \right)^{1/p}< \infty.
$$
Notice that if $\sigma \equiv 1$, then $\mathcal{A}^p_{\sigma}(\mathbb{D})=\mathcal{A}^p(D)$.
\begin{cor}\label{BPBergmanSpaces} 
A bounded set $\mathcal{F} \subseteq \mathcal{A}^p_\sigma(D) $ is precompact if and only if   
$$
    \lim_{\delta\rightarrow 0^+} \sup_{f \in \mathcal{F}} \int_{D \setminus D_\delta} |f|^p \sigma \mathop{d V}=0.
$$
\end{cor}

\begin{rem}
The hypothesis that $\mathcal{F}$ is bounded can be removed in both Corollary \ref{RadialBergmanSpaces} and Corollary \ref{BPBergmanSpaces} since the boundedness of $\mathcal{F}$ is implied by the uniformly vanishing integral condition. We illustrate the proof when $\sigma$ is a $B_p$ weight and note that the obvious modifications can be made when the weight is as in Corollary \ref{RadialBergmanSpaces}. Take $\varepsilon=1$ and fix the corresponding $\delta$ as in the proof of Theorem \ref{BesovSobolevRK} (see Section \ref{AbstractCompactnessProofs}). It suffices to show that 
$$
    \sup_{f\in\mathcal{F}}\int_{D_{\delta}}|f|^p \sigma \mathop{d V} <\infty.
$$
We claim that the functions in $\mathcal{F}$ are uniformly bounded on the compact set $\partial D_{\delta/2}$. Indeed, if $z \in \partial D_{\delta/2}$, then the Euclidean ball $B(z,\delta/4)$ is contained in $D \setminus D_\delta$. We then estimate for such a point $z$ and $f \in \mathcal{F}$ as follows
\begin{align*}
|f(z)| & \leq \frac{1}{ V(B(z,\delta/4))}\int_{B(z,\delta/4)} |f| \mathop{dV}\\
& \leq C_\delta \int_{D \setminus D_{\delta}} |f| \mathop{dV}\\
& \leq C_{\delta} \left(\int_{D \setminus D_\delta} |f|^p \sigma  \mathop{dV}\right)^{1/p} \left(\int_{D} \sigma^{-1/(p-1)} \mathop{dV}\right)^{1/p'} \\
& \leq C_{\delta,p,\sigma}.  
\end{align*}
By the maximum principle, the functions in $\mathcal{F}$ are uniformly bounded on $D_{\delta}$, and thus the above inequality holds. 
\end{rem}

\begin{rem}\label{BergmanRemark}
We note that compactness criteria for $\mathcal{A}^p(\mathbb{B}_n)$ follow from either of Theorem \ref{BesovSobolevRK} or Theorem \ref{BanachFrameRK1}. An application of Theorem \ref{BesovSobolevRK} with $D=\mathbb{B}_n$, $p \in [1,\infty)$, $\sigma \equiv \frac{1}{V(\mathbb{B}_n)}$, and $J=0$ shows that $\mathcal{F} \subseteq \mathcal{A}^p(\mathbb{B}_n)$ is precompact if and only if
\begin{equation*}
    \lim_{r \rightarrow 1^{-}} \sup_{f \in \mathcal{F}} \int_{\mathbb{B}_n \setminus r \mathbb{B}_n}|f(w)|^p \, d v(w)=0,
\label{BergmanBesov}\end{equation*} 
where $dv$ represents normalized Lebesgue measure on the unit ball. On the other hand, if $p \in (1,\infty)$, then $\mathcal{A}^p(\mathbb{B}_n)$ is a reflexive Banach space with a continuous frame $\{k_w^{(p)}, k_w^{(p')}\}$ with respect to $L^p(\mathbb{B}_n,d \lambda),$ where $k_w^{(p)}(z):= \frac{(1-|w|^2)^{\frac{n+1}{p'}}}{(1- z \overline{w})^{n+1}}$ denotes the ``$p$-normalized'' reproducing kernel at $w$ and $d\lambda (w):=(1-|w|^2)^{-(n+1)} dv(w)$ denotes the hyperbolic measure on $\mathbb{B}_n$. Theorem \ref{BanachFrameRK1} gives that $\mathcal{F}\subseteq \mathcal{A}^p(\mathbb{B}_n)$ is precompact if and only if 
\begin{equation*}
    \lim_{r \rightarrow 1^{-}} \sup_{f \in \mathcal{F}} \int_{\mathbb{B}_n \setminus r \mathbb{B}_n}|\langle f, k_w^{(p')} \rangle|^p \, d \lambda(w)=0.
\label{BergmanBanachFrame}\end{equation*} 
We also remark that Theorems \ref{HilbertFrameRK1} and \ref{HilbertFrameRK2} both apply in the Hilbert space case $p=2.$
\end{rem}

\begin{rem}
Both Corollary \ref{RadialBergmanSpaces} and Corollary \ref{BPBergmanSpaces} apply in the case of weighted Bergman spaces of $\mathbb{B}_n$ with radial weights $\sigma(z)=(1-|z|^2)^t$ for $t \in (-1,p-1)$, since $\sigma$ is a $B_p$ weight for this range of $t$. Corollary \ref{RadialBergmanSpaces} extends this fact to all $t >-1$, and Corollary \ref{BPBergmanSpaces} generalizes the result to arbitrary $B_p$ weights.
\end{rem}

The Hardy space, $\mathcal{H}^2(\mathbb{B}_n)$, is the space of holomorphic $f:\mathbb{B}_n\rightarrow\mathbb{C}$ such that
$$
    \|f\|_{\mathcal{H}^2(\mathbb{B}_n)}:= \sup_{0<r<1} \left(\int_{\partial \mathbb{B}_n} |f(r \zeta)|^2 \, d s(\zeta)\right)^{1/2}<\infty,
$$
where $s$ denotes the normalized Lebesgue surface measure on $\partial \mathbb{B}_n$. The functions in $\mathcal{H}^2(\mathbb{B}_n)$ have well-defined boundary values almost everywhere, and hence $\mathcal{H}^2(\mathbb{B}_n)$ can be isometrically identified with a closed subspace of $L^2(\partial \mathbb{B}_n)$, see \cite{Zhu}*{Theorem 4.25}. 
The Hardy space can be defined using the following equivalent norm which we also denote by $\|\cdot\|_{\mathcal{H}^2(\mathbb{B}_n)}$:
$$
        \|f\|_{\mathcal{H}^2(\mathbb{B}_n)}:= \left(\sum_{|\alpha|<J} \left|\frac{\partial^\alpha f}{\partial z^\alpha}(0)\right|^2 + \sum_{|\alpha|=J}\int_{\mathbb{B}_n} \left|\frac{\partial^\alpha f}{\partial z^\alpha}(w)\right|^2 (1-|w|^2)^{2J-1} \mathop{d v(w)}\right)^{1/2}<\infty,
$$
where $J$ is a positive integer, see \cite{ARS2008}. Note that this space is independent of $J$.

\begin{cor}\label{HardySpaces}
A bounded set $\mathcal{F}\subseteq \mathcal{H}^2(\mathbb{B}_n)$ is precompact if and only if 
$$
    \lim_{r \rightarrow 1^{-}} \sup_{f \in \mathcal{F}} \sum_{|\alpha|=J}\int_{\mathbb{B}_n \setminus r \mathbb{B}_n} \left|\frac{\partial^\alpha f}{\partial z^\alpha}(w)\right|^2 (1-|w|^2)^{2J-1} \mathop{d v(w)}=0.
$$
\end{cor}

The Besov space $\mathcal{D}^p(\mathbb{B}_n)$ is the space of holomorphic $f:\mathbb{B}_n\rightarrow\mathbb{C}$ such that 
$$
        \|f\|_{\mathcal{D}^p(\mathbb{B}_n)}:= \left(\sum_{|\alpha|<J} \left|\frac{\partial^\alpha f}{\partial z^\alpha}(0)\right|^p + \sum_{|\alpha|=J}\int_{\mathbb{B}_n} \left|\frac{\partial^\alpha f}{\partial z^\alpha}(w)\right|^p (1-|w|^2)^{pJ-(n+1)} \mathop{d v(w)}\right)^{1/p}<\infty,
$$
where $J > \frac{n}{p}$ is an integer. These are precisely the scale of Besov spaces discussed in \cite{Zhu}*{Chapter 6}, and in the special case $p=2,$ $\mathcal{D}^2(\mathbb{B}_n)$ is the well-known Dirichlet space. As with the Hardy space, these Besov spaces do not depend on the choice of $J$.

\begin{cor}\label{DirichletSpaces}
A bounded set $\mathcal{F}\subseteq \mathcal{D}^p(\mathbb{B}_n)$ is precompact if and only if  
$$
    \lim_{r \rightarrow 1^{-}} \sup_{f \in \mathcal{F}} \sum_{|\alpha|=J}\int_{\mathbb{B}_n \setminus r \mathbb{B}_n} \left|\frac{\partial^\alpha f}{\partial z^\alpha}(w)\right|^p (1-|w|^2)^{pJ-(n+1)} \mathop{d v(w)}=0.
$$
\end{cor}


\subsection{Applications of compactness criteria}
We apply our compactness criteria to characterize the compact Toeplitz operators on the Bergman space, deduce the compactness of Hankel operators on the Hardy space, and obtain general umbrella theorems. We have chosen to provide only a sampling of the possible applications of our results. It is clear that more could be done including working with different frames, extending applications outside of $L^2$/$L^p$ settings, and obtaining additional operator theoretic applications; however, we aim to provide just a flavor of the possible applications. 

\subsubsection{Compactness of Toeplitz operators on the Bergman space}
Our first application is a characterization of the compact Toeplitz operators on the Bergman space of the unit ball. Given a function $u$ on $\mathbb{B}_n$, the Toeplitz operator assosociated to $u$, $T_u$, is given by
$$
    T_uf(z):=P(uf)(z)=\int_{\mathbb{B}_n}\frac{1}{(1-z\overline{w})^{n+1}}u(w)f(w)\,dv(w),
$$
where $P$ denotes the Bergman projection from $L^2(\mathbb{B}_n)$ onto $\mathcal{A}^2(\mathbb{B}_n)$ and $z\overline{w}=\sum_{j=1}^nz_j\overline{w}_j$. Below, 
$\widetilde{T}$ represents the Berezin transform of a bounded operator $T$ on $\mathcal{A}^p(\mathbb{B}_n)$ defined by
$$
    \widetilde{T}(z):=\langle Tk_z, k_z\rangle_{\mathcal{A}^2(\mathbb{B}_n)},
$$ 
where $k_z:=k_z^{(2)}$ is the normalized reproducing kernel of $\mathcal{A}^2(\mathbb{B}_n)$ at $z$. 
\begin{thm}\label{ToeplitzCompactness}
Let $T$ be a finite sum of finite products of Toeplitz operators with $L^\infty(\mathbb{B}_n)$ symbols and $p \in (1,\infty)$. Then $T$ is compact on $\mathcal{A}^p(\mathbb{B}_n)$ if and only if 
$$
    \lim_{|z| \rightarrow 1^-}\widetilde{T}(z)=0.
$$
\end{thm}

Theorem \ref{ToeplitzCompactness} contains the seminal result of Axler and Zheng from \cite{AZ1998}. 
Theorem \ref{ToeplitzCompactness} has recently been extended strongly pseudoconvex domains with smooth boundary by Wang and Xia in \cite{WX2021}*{Proposition 9.3}, however our methods are different and considerably less involved. See also \cites{IMW2015,MSW2013,S2007} for related results. Additionally, we mention that the analogous result for smoothly bounded strongly pseudoconvex domains can be obtained with our methods using our Corollary \ref{RadialBergmanSpaces} or Corollary \ref{BPBergmanSpaces}.

\subsubsection{Compactness of little Hankel operators on the Hardy space}
Our second application deals with the compactness of little Hankel operators on the Hardy space. 
Given $g \in \mathcal{H}^2(\mathbb{B}_n)$, the little Hankel operator, $H_g$, is given by 
$$
    H_gf(z):=S(g \overline{f})(z)=\int_{\partial\mathbb{B}_n}\frac{1}{(1-z\overline{w} )^n}g(w)\overline{f(w)}\,ds(w),
$$
where $S$ denotes the Szeg\H{o} projection from $L^2(\partial \mathbb{B}_n)$ onto $\mathcal{H}^2(\mathbb{B}_n)$. 
Below, VMOA represents the space of $f\in \mathcal{H}^2(\mathbb{B}_n)$ with vanishing mean oscillation, that is

$$
    \lim_{r \rightarrow 0^{+}} \sup_{\zeta \in \partial \mathbb{B}_n} \frac{1}{s(Q(\zeta,r))} \int_{Q(\zeta, r)} \left|f(w)-\frac{1}{s(Q(\zeta,r))}\int_{Q(\zeta,r)} f\,ds \right|^2 \,ds(w)=0,
$$
where $Q(\zeta,r)$ denotes to the ball in $\partial \mathbb{B}_n$ centered at $\zeta$ of radius $r$ with respect to the non-isotropic metric $d(z,w)=|1- z\overline{w}|^{1/2}$, see \cite{Zhu}. 
\begin{thm}\label{HankelCompactness}
Let $g \in \mathcal{H}^2( \mathbb{B}_n)$. The little Hankel operator $H_g$ is compact on $\mathcal{H}^2( \mathbb{B}_n)$ if and only if $g \in \text{VMOA}.$
\end{thm}

This result first appeared in \cite{Hartman1958} in the 1-dimensional setting and appears in \cite{CRW} for the unit ball in $\mathbb{C}^n$. Similar questions have been considered for more general commutators of Calder\'{o}n-Zygmund operators in \cite{U1978}.  Our framework is best suited to proving the sufficiency of VMOA for compactness -- we only supply a proof for this direction of the theorem.

\subsubsection{General umbrella theorems} 

The following is a form of uncertainty principle of Fourier analysis known as Shapiro's umbrella theorem that follows quickly from the classical Riesz-Kolmogorov theorem on $L^2(\mathbb{R})$, or more precisely from Theorem~C above, see \cite{JP2007}: 
``Let $\varphi, \psi \in L^2(\mathbb{R})$. If $\{e_k\} \subseteq L^2(\mathbb{R})$ is an orthonormal sequence of functions such that for each $k$ and almost every $x,\xi \in \mathbb{R}$, 
$$
    |e_k(x)|\leq|\varphi(x)| \quad\quad\text{and}\quad\quad|\widehat{e_k}(\xi)|\leq|\psi(\xi)|,
$$
then $\{e_k\}$ is finite." In other words, no infinite orthonormal sequence of $L^2$ functions can have common umbrella functions $\varphi, \psi \in L^2(\R)$. 

The orthonormality assumption in the umbrella theorem can be replaced with many weaker conditions, such as separation, being a Bessel sequence, being a frame for its closed span, being a Schauder basis for its closed span, etc.. 
Any sequence of vectors having a common umbrella, due to Riesz-Kolmogorov type criteria is forced to be compact, and consequently to have a convergent subsequence, which is not possible for any of the above mentioned types of sequences. With this in mind, it is clear that each of our compactness criteria will imply a corresponding umbrella theorem. We only state the one for Besov-Sobolev spaces.

\begin{thm}
Let $D \subseteq \mathbb{C}^n$ be a bounded domain and $\sigma$ be an integrable weight on $D$. 
Let $\FF\subseteq \mathcal{B}_{\sigma}^{p,J}(D)$ be a separated family of functions, that is, there exists $\delta>0$ such that $\|f-g\|_{\mathcal{B}_{\sigma}^{p,J}(D)}\geq \delta$ for all distinct $f, g\in \FF$. If there exists $\varphi \in L^p_\sigma(D)$ such that 
\[
    \left|\frac{\partial^\alpha f}{\partial z^\alpha}(z)\right|\leq \varphi(z)
\]
for all $f\in \mathcal{F}$, $z\in D$, and $|\alpha|=J$, then $\FF$ is a finite set.
\end{thm}




The remainder of the paper is organized as follows. We prove our precompactness characterizations Theorem \ref{HilbertRK}, Theorem \ref{HilbertFrameRK1}, Theorem \ref{HilbertFrameRK2}, Theorem \ref{BanachFrameRK1}, Theorem \ref{BanachMeanValueRK}, Theorem \ref{BesovSobolevRK}, Corollary \ref{RadialBergmanSpaces},
Corollary \ref{BPBergmanSpaces}, Corollary \ref{HardySpaces}, and Corollary \ref{DirichletSpaces} 
in Section \ref{AbstractCompactnessProofs}. We then prove the characterization of compact Toeplitz operators on $\mathcal{A}^p(\mathbb{B}_n)$, 
Theorem \ref{ToeplitzCompactness}, and the compactness of little Hankel operators on $\mathcal{H}^2(\mathbb{B}_n)$, Theorem \ref{HankelCompactness}, in Section \ref{BergmanSection}. 


\section{Proofs of compactness criteria} \label{AbstractCompactnessProofs}


\subsection{General compactness characterizations}

\begin{proof}[Proof of Theorem \ref{HilbertRK}]
To prove the forward direction, we suppose $\mathcal{F}$ is precompact and proceed by contradiction. Assuming that the uniform decay condition fails, there exists $\varepsilon_0>0$ such that for each $n\in\mathbb{N}$, there exists $f_n\in \mathcal{F}$ with 
$$
    |\langle T_nf_n-f_n,f_n\rangle_{\mathcal{H}}|> \varepsilon_0.
$$
Consider the sequence $\{f_n\}_{n=1}^{\infty}$. By the precompactness of $\mathcal{F}$, there exists a subsequence $\{f_{n_k}\}_{k=1}^{\infty}$ converging to some $f$ in $\mathcal{H}$. We claim that for all $k>0$, we have 
$$
    |\langle T_{n_k}f-f,f\rangle_{\mathcal{H}}|> \frac{\varepsilon_0}{2},
$$
providing a contradiction since $\lim_{n\rightarrow\infty}\langle T_nf-f,f\rangle_{\mathcal{H}}=0$. 

In order to justify the claim, we first note that the condition $\lim_{n\rightarrow \infty}\langle T_nf-f,f\rangle_{\mathcal{H}}=0$ for all $f \in \mathcal{H}$ and polarization imply that 
\begin{align*}
   \lim_{n\rightarrow \infty} \langle T_nf,g\rangle_{\mathcal{H}}
   &=\lim_{n\rightarrow\infty}\frac{1}{4}\sum_{k=0}^{3}i^k\langle T_n(f+i^kg),f+i^kg\rangle_{\mathcal{H}}\\
   &=\frac{1}{4}\sum_{k=0}^{3}i^k\langle f+i^kg,f+i^kg\rangle_{\mathcal{H}}\\
   &= \langle f,g\rangle_{\mathcal{H}}
\end{align*}
for all $f, g \in \mathcal{H}$, which implies that the $T_n$ are uniformly bounded. Indeed, the above weak operator topology convergence of the $T_n$ implies that the linear functionals $T_{n,f}:\mathcal{H}\rightarrow\mathbb{C}$ given by $T_{n,f}g=\langle T_nf, g\rangle_{\mathcal{H}}$ satisfy
$$
    \sup_{n>0} |T_{n,f}g|=\sup_{n>0}|\langle T_nf,g\rangle_{\mathcal{H}}|<\infty
$$
for all $f,g\in\mathcal{H}$. By the uniform boundedness principle, 
$$
    \sup_{n>0}\|T_nf\|_{\mathcal{H}}=\sup_{n>0}\|T_{n,f}\|_{\mathcal{H}\rightarrow\mathbb{C}}<\infty
$$
for all $f \in \mathcal{H}$, and therefore, by another application of the uniform boundedness principle,
$$
    B:=\sup_{n>0}\|T_n\|_{\mathcal{H}\rightarrow\mathcal{H}}<\infty.
$$
Write $A:=\sup_{k>0}\|f_{n_k}\|_{\mathcal{H}}$ and choose $k>0$ large enough such that
$$
    \|f_{n_k}-f\|_{\mathcal{H}}<\max\left(\frac{\varepsilon_0}{4(B+1)\|f\|_{\mathcal{H}}},\frac{\varepsilon_0}{4(B+1)A}\right).
$$
Using the reverse triangle inequality, we have for each $k>0$ that 
\begin{align*}
    |\langle T_{n_k}f-f,f\rangle_{\mathcal{H}}|&\ge |\langle T_{n_k}f_{n_k}-f_{n_k},f_{n_k}\rangle_{\mathcal{H}}|-|\langle T_{n_k}f-f,f\rangle_{\mathcal{H}}-\langle T_{n_k}f_{n_k}-f_{n_k},f_{n_k}\rangle_{\mathcal{H}}|\\
    &> \varepsilon_0-(|\langle T_{n_k}f-f,f-f_{n_k}\rangle_{\mathcal{H}}|+|\langle T_{n_k}(f-f_{n_k})-(f-f_{n_k}),f_{n_k}\rangle_{\mathcal{H}}|).
\end{align*}
The claim holds since, the  Cauchy-Schwarz inequality implies
\begin{align*}
    |\langle T_{n_k}f-f,f-f_{n_k}\rangle_{\mathcal{H}}|
    &\leq \|T_{n_k}f-f\|_{\mathcal{H}}\|f-f_{n_k}\|_{\mathcal{H}}\\
    &\leq (B+1)\|f\|_{\mathcal{H}}\|f-f_{n_k}\|_{\mathcal{H}}\\
    &< (B+1)\|f\|_{\mathcal{H}}\left(\frac{\varepsilon_0}{4(B+1)\|f\|_{\mathcal{H}}}\right)=\frac{\varepsilon_0}{4}
\end{align*}
and similarly
\begin{align*}
    |\langle T_{n_k}(f-f_{n_k})-(f-f_{n_k}),f_{n_k}\rangle_{\mathcal{H}}|
    &\leq\|T_{n_k}(f-f_{n_k})-(f-f_{n_k})\|_{\mathcal{H}}\|f_{n_k}\|_{\mathcal{H}}\\
    &\leq (B+1)A\|f-f_{n_k}\|_{\mathcal{H}}\\
    &< (B+1)A\left(\frac{\varepsilon_0}{4(B+1)A}\right)=\frac{\varepsilon_0}{4}.
\end{align*}

To establish the reverse direction, let $\{f_j\}_{j=1}^{\infty}\subseteq \mathcal{F}$. Since $\mathcal{F}$ is bounded, there exists a subsequence $\{f_{j_k}\}_{k=1}^{\infty}$ converging weakly to some $f\in\mathcal{H}$; we claim that $\{f_{j_k}\}_{k=1}^{\infty}$ converges strongly to $f$. Let $\varepsilon>0$ and apply the hypotheses to choose $N>0$ such that 
$$
    \sup_{k>0}|\langle T_Nf_{j_k}-f_{j_k},f_{j_k}\rangle_{\mathcal{H}}|<\frac{\varepsilon}{4} \quad\quad\text{and}\quad\quad |\langle T_Nf-f,f\rangle_{\mathcal{H}}|<\frac{\varepsilon}{4}.
$$ 
By the compactness of $T_{N}$, we have that $\{T_{N}f_{j_k}\}_{k=1}^{\infty}$ converges strongly to $T_{N}f$, which implies that $\langle T_{N}f_{j_k}-f,f_{j_k}\rangle_{\mathcal{H}} \rightarrow \langle T_Nf-f,f\rangle_{\mathcal{H}}$ as $k\rightarrow \infty$. Take $K>0$ such that 
$$
    |\langle T_Nf_{j_k}-f,f_{j_k}\rangle_{\mathcal{H}}|<|\langle T_Nf-f,f\rangle_{\mathcal{H}}|+\frac{\varepsilon}{4} \quad\quad\text{and}\quad\quad |\langle f_{j_k}-f,f\rangle_{\mathcal{H}}|<\frac{\varepsilon}{4}
$$
for all $k> K$. Then
\begin{align*}
    \|f_{j_k}-f\|_{\mathcal{H}}^2 &= - \langle T_Nf_{j_k}-f_{j_k},f_{j_k}\rangle_{\mathcal{H}}+\langle T_Nf_{j_k}-f_{j_k},f_{j_k}\rangle_{\mathcal{H}}+\langle f_{j_k}-f,f_{j_k}-f\rangle_{\mathcal{H}}\\
    &\leq |\langle T_Nf_{j_k}-f_{j_k},f_{j_k}\rangle_{\mathcal{H}}|+|\langle T_Nf_{j_k}-f,f_{j_k}\rangle_{\mathcal{H}}|+|\langle f_{j_k}-f,f\rangle_{\mathcal{H}}|\\
    &< |\langle T_Nf_{j_k}-f_{j_k},f_{j_k}\rangle_{\mathcal{H}}|+|\langle T_Nf-f,f\rangle_{\mathcal{H}}|+\frac{\varepsilon}{4}+|\langle f_{j_k}-f,f\rangle_{\mathcal{H}}|\\
    &< \frac{\varepsilon}{4}+\frac{\varepsilon}{4}+\frac{\varepsilon}{4}+\frac{\varepsilon}{4}=\varepsilon
\end{align*}
for all $k>K$. 
\end{proof}

\begin{proof}[Proof of Theorem \ref{HilbertFrameRK1}]
The proof of the forward direction is similar to the corresponding implication in the proof of Theorem \ref{HilbertRK}. Suppose $\mathcal{F}$ is precompact and assume towards contradiction that there exists $\varepsilon_0>0$ such that for each $n \in \mathbb{N}$, there exists $f_n \in \mathcal{F}$ with 
$$
    \int_{X \setminus F_n}|\langle f_n,k_x\rangle_{\mathcal{H}}|^2\,d\mu(x)>\varepsilon_0.
$$
Since $\mathcal{F}$ is precompact, there exists a subsequence $\{f_{n_k}\}_{k=1}^{\infty}$ of $\{f_n\}_{n=1}^{\infty}$ that converges to some $f$ in $\mathcal{H}$. We claim that 
$$
    \int_{X\setminus F_{n_k}}|\langle f,k_x\rangle_{\mathcal{H}}|^2\,d\mu(x) > \frac{\varepsilon_0}{4}
$$
for all $k \in \mathbb{N}$, providing a contradiction since $f \in \mathcal{H}$. Choose $k \in \mathbb{N}$ large enough such that $\|f-f_{n_k}\|_{\mathcal{H}}<\frac{\varepsilon_0^{1/2}}{2}$. Then applying the reverse triangle inequality, we have
\begin{align*}
    \left(\int_{X\setminus F_{n_k}}|\langle f,k_x\rangle_{\mathcal{H}}|^2\,d\mu(x)\right)^{1/2}&\ge \left(\int_{X \setminus F_{n_k}}|\langle f_{n_k},k_x\rangle_{\mathcal{H}}|^2\,d\mu(x)\right)^{1/2}\\
    &\quad\quad-\left(\int_{X \setminus F_{n_k}}|\langle f-f_{n_k},k_x\rangle_{\mathcal{H}}|^2\,d\mu(x)\right)^{1/2}\\
    &\ge \varepsilon_0^{1/2}-\|f-f_{n_k}\|_{\mathcal{H}}\\
    &> \frac{\varepsilon_0^{1/2}}{2}.
\end{align*}

To prove the other direction, let $\{f_j\}_{j=1}^{\infty}\subseteq \mathcal{F}$. Since $\mathcal{F}$ is bounded, there exists a subsequence $\{f_{j_k}\}_{k=1}^{\infty}$ that converges weakly. Without loss of generality, assume that $\{f_{j_k}\}_{k=1}^{\infty}$ converges weakly to $0$. We claim that $\{f_{j_k}\}_{k=1}^{\infty}$ converges strongly to $0$. Let $\varepsilon>0$ and apply the condition to choose $N \in \mathbb{N}$ such that $\int_{X\setminus F_N}|\langle f_{j_k},k_x\rangle_{\mathcal{H}}|^2\,d\mu(x) < \frac{\varepsilon}{2}$ for all $k$. Then
\begin{align*}
    \|f_{j_k}\|_{\mathcal{H}}^2 &= \int_X |\langle f_{j_k},k_x\rangle_{\mathcal{H}}|^2\,d\mu(x)\\
    &= \int_{F_N} |\langle f_{j_k},k_x\rangle_{\mathcal{H}}|^2\,d\mu(x)+\int_{X\setminus F_N} |\langle f_{j_k},k_x\rangle_{\mathcal{H}}|^2\,d\mu(x)\\
    &<\int_{F_N} |\langle f_{j_k},k_x\rangle_{\mathcal{H}}|^2\,d\mu(x)+\frac{\varepsilon}{2}.
\end{align*}
Now, $\int_{F_N} |\langle f_{j_k},k_x\rangle_{\mathcal{H}}|^2\,d\mu(x) \leq \mu(F_N)\sup_{f\in\mathcal{F}}\|f\|_{\mathcal{H}}$, and therefore, we may use the fact that the $f_{j_k}$ converge weakly to $0$ and apply the the dominated convergence theorem to choose $K \in \mathbb{N}$ such that $\int_{F_N}|\langle f_{j_k},k_x\rangle_{\mathcal{H}}|^2\,d\mu(x) < \frac{\varepsilon}{2}$ for all $k>K$. This establishes the result.
\end{proof}

\begin{proof}[Proof of Theorem \ref{HilbertFrameRK2}]

Consider the operators $T_{F_n}:\mathcal{H}\rightarrow\mathcal{H}$ given by
$$
    T_{F_n}f=\int_{F_n}\langle f,k_x\rangle_{\mathcal{H}}k_x\,d\mu(x).
$$ 
Note that
$$
    \langle T_{F_n}f-f,f\rangle_{\mathcal{H}}=-\left\langle \int_{X\setminus F_{n}} \langle f,k_x\rangle_{\mathcal{H}} k_x\,d\mu(x), f\right\rangle_{\mathcal{H}}=-\int_{X\setminus F_n}|\langle f,k_x\rangle_{\mathcal{H}}|^2\,d\mu(x).
$$
This identity implies that $\langle T_{F_n}f-f,f\rangle_{\mathcal{H}}\rightarrow 0$ as $n\rightarrow \infty$ for each $f \in \mathcal{H}$ (by the dominated convergence theorem) and that the condition 
$$
    \lim_{n\rightarrow \infty}\sup_{f \in \mathcal{F}} \langle T_{F_n}f-f,f\rangle_{\mathcal{H}}=0
$$
is equivalent to 
$$
    \lim_{n\rightarrow\infty}\sup_{f\in\mathcal{F}}\int_{X\setminus F_{n}}|\langle f,k_x\rangle_{\mathcal{H}}|^2\,d\mu(x)=0.
$$
Appealing to \cite{GM2021}*{Proposition 1 and Theorem 1} by the hypotheses of on the frame $\{k_x\}$, each $T_{F_n}$ is compact on $\mathcal{H}$. The theorem follows upon applying Theorem \ref{HilbertRK} with $T_n=T_{F_n}$.

\end{proof}

\begin{proof}[Proof of Theorem \ref{BanachFrameRK1}]
The proof is very similar to the proof of Theorem \ref{HilbertFrameRK1}. We omit the details of the forward direction.


To show the reverse direction. Let $\{f_j\}_{j=1}^{\infty}\subseteq \mathcal{F}$. Since $\mathcal{F}$ is bounded and since $\mathcal{X}$ is reflexive, there exists a subsequence $\{f_{j_k}\}_{k=1}^{\infty}$ that converges weakly. Without loss of generality, assume that $\{f_{j_k}\}_{k=1}^{\infty}$ converges weakly to $0$. We claim that $\{f_{j_k}\}_{k=1}^{\infty}$ converges strongly to $0$. Let $\varepsilon>0$ and apply the condition to choose $N \in \mathbb{N}$ such that $\int_{X\setminus F_N}|\langle f_{j_k},f_x^*\rangle|^p\,d\mu(x) < c^p\frac{\varepsilon}{2}$ for all $k$. Then
\begin{align*}
    \|f_{j_k}\|_{\mathcal{X}}^p &\leq \frac{1}{c^p}\int_X |\langle f_{j_k},f_x^*\rangle|^p\,d\mu(x)\\
    &= \frac{1}{c^p}\int_{F_N} |\langle f_{j_k},f_x^*\rangle|^p\,d\mu(x)+\frac{1}{c^p}\int_{X\setminus F_N} |\langle f_{j_k},f_x^*\rangle|^p\,d\mu(x)\\
    &<\frac{1}{c^p}\int_{F_N} |\langle f_{j_k},f_x^*\rangle|^p\,d\mu(x)+\frac{\varepsilon}{2}.
\end{align*}
Since $\mathcal{F}$ is bounded, $\int_{F_N} |\langle f_{j_k},f_x^*\rangle|^p\,d\mu(x) \leq \mu(F_N)(\sup_{x\in X}\|f_x^*\|_{\mathcal{X}\rightarrow\mathbb{C}}\sup_{f \in \mathcal{F}}\|f\|_{\mathcal{X}})^p$. Therefore, we may use the fact that the $f_{j_k}$ converge weakly to $0$ and apply the the dominated convergence theorem to choose $K \in \mathbb{N}$ such that $\int_{F_N}|\langle f_{j_k},f_x^*\rangle|^p\,d\mu(x) < c^p\frac{\varepsilon}{2}$ for all $k>K$. This establishes the result.
\end{proof}

\begin{proof}[Proof of Theorem \ref{BanachMeanValueRK}]
The forward direction follows from making slight modifications to the argument in the proof of Theorem \ref{HilbertFrameRK1}. We omit the details.

To show the reverse direction, let $\{f_n\}_{n=1}^{\infty}\subseteq \mathcal{F}$ be a sequence. Consider $\{D_1f_n\chi_{F_{m}}\}_{n=1}^{\infty}\subseteq C(F_m)$ for a fixed $m \in \mathbb{N}$. 
This sequence is equicontinuous and pointwise uniformly bounded (due to hypotheses (ii) and (iii)), and therefore has a subsequence that converges in the uniform norm on $F_m$ by the Arzel\`a-Ascoli theorem. 
Denote the uniformly convergent subsequence $\{D_1f_{n_k}\chi_{F_m}\}_{k=1}^{\infty}$. 

Repeating the above argument with the sequence $\{D_2f_{n_k}\chi_{F_m}\}_{k=1}^{\infty}$ we obtain indices (which we also denote by $n_k$) such that $\{D_1f_{n_k}\chi_{F_m}\}_{k=1}^{\infty}$ and $\{D_2f_{n_k}\chi_{F_m}\}_{k=1}^{\infty}$ converge uniformly on $F_m$. Repeating this process $N$ times, we obtain a subsequence $\{f_{n_k}\}_{k=1}^{\infty}$ of $\{f_n\}_{n=1}^{\infty}$ such that $\{D_jf_{n_k}\chi_{F_m}\}_{k=1}^{\infty}$ converges uniformly on $F_m$ for each $1\leq j\leq N$. A standard diagonalization argument then allows one to extract a further subsequence (again denoted by $\{f_{n_k}\}_{k=1}^{\infty}$) such that $\{D_jf_{n_k}\chi_{F_m}\}_{k=1}^{\infty}$ converges uniformly on $F_m$ for all $1\leq j\leq N$ and all 
$m \in \mathbb{N}$. Since $\{D_{N+1}f_{n_k}(x_0)\}_{k=1}^{\infty}$ is a bounded sequence of complex numbers, it has a convergent subsequence (also indexed by $n_k$) by the Bolzano-Weierstrass theorem. Iterating this argument $M$ times yields a subsequence $\{f_{n_k}\}_{k=1}^{\infty}$ of $\{f_n\}_{n=1}^{\infty}$ such that $\{D_j(f_{n_k})\chi_{F_m}\}_{k=1}^{\infty}$ converges uniformly on $F_m$ for all $1\leq j\leq N$ and all 
$m \in \mathbb{N}$ and $\{D_{j}f_{n_k}(x_0)\}_{k=1}^{\infty}$ converges for all $N+1\leq j \leq N+M$.

We claim that $\{f_{n_k}\}_{k=1}^{\infty}$ is Cauchy in $\mathcal{X}$. To see this, let $\varepsilon >0$. By hypothesis, we may choose $M_0 \in \mathbb{N}$ such that $\sum_{j=1}^{N}\int_{X\setminus F_{M_0}}|D_jf_{n_k}|^p\,d\mu<\frac{\varepsilon}{3\cdot 2^{p+1}}$ for all $k \in \mathbb{N}$. Let $N_0 \in \mathbb{N}$ be such that for all $k,\ell >N_0$, $|D_jf_{n_k}(x)-D_jf_{n_{\ell}}(x)|<\left(\frac{\varepsilon}{3N\mu(F_{M_0})}\right)^{1/p}$ for all $x \in F_{M_0}$ and all $1 \leq j \leq N$, and $|D_jf_{n_k}(x_0)-D_jf_{n_{\ell}}(x_0)|< \left(\frac{\varepsilon}{3M}\right)^{1/p}$ for all $N+1\leq j\leq N+M$. For $k,\ell>N_0$, we have 
\begin{align*}
	\|f_{n_k}-f_{n_{\ell}}\|_{\mathcal{X}}^p&=\sum_{j=1}^N\int_{X\setminus F_{M_0}}|D_j(f_{n_k})-D_j(f_{n_{\ell}})|^p\,d\mu+\sum_{j=1}^N\int_{F_{M_0}}|D(f_{n_k})-D_j(f_{n_\ell})|^p\,d\mu\\
	&\quad\quad+\sum_{j=N+1}^{N+M}|D_jf_{n_k}(x_0)-D_jf_{n_{\ell}}(x_0)|^p\\
	&< 2^p\sum_{j=1}^N\int_{X\setminus F_{M_0}}|D_j(f_{n_k})|^p\, d\mu+2^p\sum_{j=1}^N\int_{X\setminus F_{M_0}}|D_j(f_{n_\ell})|^p\,d\mu\\
	&\quad\quad+\sum_{j=1}^N\int_{F_{M_0}}\left(\frac{\varepsilon}{3N\mu(F_{M_0})}\right)\,d\mu +\sum_{j=N+1}^{N+M}\left(\frac{\varepsilon}{3M}\right)\\
	&<2^p\frac{\varepsilon}{3\cdot 2^{p+1}}+2^p\frac{\varepsilon}{3\cdot 2^{p+1}}+N\frac{\varepsilon}{3N\mu(F_{M_0})}\mu(F_{M_0})+M\frac{\varepsilon}{3M} = \varepsilon.
\end{align*}
This proves our claim and therefore establishes the theorem.
\end{proof}

\subsection{Weighted Besov-Sobolev space compactness characterization}\label{BesovSobolevSubsection}


We begin by establishing that the weighted Besov-Sobolev spaces, $\mathcal{B}_{\sigma}^{p,J}(D)$, are in fact Banach spaces. 

\begin{prop}\label{BanachSpace} If $D\subseteq \mathbb{C}^n$ is a bounded domain, $p \in [1,\infty)$, $J \in \mathbb{N}$, and $\sigma$ is an integrable weight, then $\mathcal{B}_{\sigma}^{p,J}(D)$ is a Banach space. 
\begin{proof}
It suffices to prove completeness. Suppose $\{f_j\}_{j=1}^{\infty}$ is Cauchy in  $\mathcal{B}_{\sigma}^{p,J}(D)$. We claim that $\{f_j\}_{j=1}^{\infty}$ is uniformly Cauchy on compact subsets of $D$. Assuming the claim, then $\{f_j\}_{j=1}^{\infty}$ converges pointwise to a function $g$, and since the convergence is uniform on compact subsets, $g$ is holomorphic and we also have convergence of the derivatives. Moreover, for each $\alpha$ with $|\alpha|=J,$ we know that $\{\frac{\partial^\alpha f_j}{\partial z^\alpha}\}_{j=1}^{\infty}$ converges in $L^p_\sigma(D)$ to a function $h_\alpha$. We also know that $\{\frac{\partial^\alpha f_j}{\partial z^\alpha}\}_{j=1}^{\infty}$ converges pointwise to $\frac{\partial^\alpha g}{\partial z^\alpha},$ so in fact $h_\alpha= \frac{\partial^\alpha g}{\partial z^\alpha}.$ It follows that $\{f_j\}_{j=1}^{\infty}$ converges to $g$ in $\mathcal{B}_{\sigma}^{p,J}(D).$

 It remains to establish the claim. We first consider the special case where $D$ is a star domain with respect to the point $z_0$. Fix a compact set $K\subseteq D$. 
 Note that for each $\alpha$ with $|\alpha|=J$ and any $z \in K$, we have, by hypothesis \eqref{BergmanInequalityCondition}, that
 $$
    \left| \frac{ \partial^\alpha f_j}{\partial z^\alpha}(z)-\frac{ \partial^\alpha f_k}{\partial z^\alpha}(z)\right|^p \leq C_{K,\sigma,p} \int_{D} \left| \frac{\partial^\alpha f_j}{\partial z^\alpha}-\frac{\partial^\alpha f_k}{\partial z^\alpha}\right|^p \sigma \mathop{dV}. 
 $$
 Since the right hand side is independent of $z$ and vanishes as $j,k\rightarrow\infty$, we conclude that for each $\alpha$ with $|\alpha|=J$, the sequence $\{\frac{\partial^\alpha f_j}{\partial z^\alpha}\}_{j=1}^{\infty}$ is uniformly Cauchy on $K.$

 Next, take a multi-index $\beta$ with $|\beta|=J-1.$ Note that our assumption implies the sequence of complex numbers $\{\frac{\partial^\beta f_j}{\partial z^\beta}(z_0)\}_{j=1}^{\infty}$ is Cauchy. For ease of notation, let 
 $$
     \frac{\partial^\beta f_j}{\partial z^\beta}-\frac{\partial^\beta f_k}{\partial z^\beta}:=F_{j,k}^{\beta}=G_{j,k}^\beta+\mathrm{i} H_{j,k}^\beta,
 $$ 
 where $G$ and $H$ are real valued functions. We estimate as follows, applying the real-variable mean value theorem on $\mathbb{C}^n=\mathbb{R}^{2n}$ (here $\nabla$ denotes the real gradient):

 \begin{align*}
 & \left|\frac{\partial^\beta f_j}{\partial z^\beta}(z)-\frac{\partial^\beta f_k}{\partial z^\beta}(z)\right|  = \left|F_{j,k}^\beta(z)\right|\\
 & \quad\leq |G_{j,k}^\beta(z)|+ |H_{j,k}^\beta(z)|\\
 & \quad\leq |G_{j,k}^\beta(z_0)| + \sup_{w \in K}|\nabla G(w)||z-z_0|+ |H_{j,k}^\beta(z_0)| + \sup_{w \in K}|\nabla H(w)||z-z_0|\\
 & \quad\lesssim |F_{j,k}^\beta(z_0)|+\sup_{w \in K} \left(\sum_{|\alpha|=J} \left| \frac{\partial^\alpha f_j}{\partial z^\alpha}(w)-\frac{\partial^{\alpha} f_k}{\partial z^{\alpha}}(w) \right|^2 \right)^{1/2}|z-z_0|\\
 & \quad= \left|\frac{\partial^\beta f_j}{\partial z^\beta}(z_0)-\frac{\partial^\beta f_k}{\partial z^\beta}(z_0) \right|+\sup_{w \in K} \left(\sum_{|\alpha|=J} \left| \frac{\partial^\alpha f_j}{\partial z^\alpha}(w)-\frac{\partial^{\alpha} f_k}{\partial z^{\alpha}}(w) \right|^2 \right)^{1/2}|z-z_0|.
 \end{align*}
 It follows that the sequence of functions $\{\frac{\partial^\beta f_j}{\partial z^\beta}\}_{j=1}^{\infty}$ is uniformly Cauchy on $K$. This argument can be iterated until we finally obtain that $\{f_j\}_{j=1}^{\infty}$  is uniformly Cauchy on $K$. Since $K$ was an arbitrary compact set, we are done. 

 We next prove the claim in the case where $D$ is not a star domain. Fix a compact set $K\subseteq D$, $z \in K,$ and assume without loss of generality that $K$ contains $z_0$ and is path-connected. Because $K$ is compact, it can be covered by finitely many Euclidean balls, where the number of balls depends only on $K.$ Replacing $K$ by a potentially larger compact set, we can assume that $K$ is equal to the finite union of the closed balls.  Let $\gamma$ be a simple path connecting $z \in K$ and $z_0.$  In particular, we can assume that the path $\gamma$ passes through each ball in the finite cover at most once. Construct a piece-wise linear path between $z$ and $z_0$ with finitely many segments by connecting the centers of these balls with the boundary points that $\gamma$ intersects, and this piece-wise linear path remains in $K.$ For each line segment, we can apply the same estimates as above to control $\left|\frac{\partial^\beta f_j}{\partial z^\beta}(\cdot)-\frac{\partial^\beta f_k}{\partial z^\beta}(\cdot)\right|.$ Iterate to obtain
 $$
     \left|\frac{\partial^\beta f_j}{\partial z^\beta}(z)-\frac{\partial^\beta f_k}{\partial z^\beta}(z)\right| \lesssim \left|\frac{\partial^\beta f_j}{\partial z^\beta}(z_0)-\frac{\partial^\beta f_k}{\partial z^\beta}(z_0) \right|+C_K \sup_{w \in K} \left(\sum_{|\alpha|=J} \left| \frac{\partial^\alpha f_j}{\partial z^\alpha}(w)-\frac{\partial^{\alpha} f_k}{\partial z^{\alpha}}(w) \right|^2 \right)^{1/2},
 $$
where $C_K$ is a constant that only depends on $K$. The remainder of the proof is as before. 
\end{proof}
\end{prop}

\begin{proof}[Proof of Theorem \ref{BesovSobolevRK}]

 We apply Theorem \ref{BanachMeanValueRK}. In the notation of Theorem \ref{BanachMeanValueRK}, we have $\mathcal{X}=\mathcal{B}_{\sigma}^{p,J}(D)$, $d$ is the Euclidean metric, and $\mu=B$ is the induced Lebesgue measure.  We can choose $\{F_m\}_{m=1}^{\infty}$ to be any compact exhaustion of $D$ and $x_0=z_0 \in D$ any arbitrary fixed point with respect to which we compute the $\mathcal{B}^{p,J}_{\sigma}$ norm. In this case, the maps $D_j$ for $1 \leq j \leq N$ are the $J^{\text{th}}$ order complex partial derivatives, that is, the maps $f \mapsto \partial^\alpha f$ where $|\alpha|=J$ (in particular, there are $N=\binom{n+J-1}{n-1}$ such maps). The maps $D_j$ for $N+1 \leq j \leq N+M$ are all the lower order derivatives, that is, all the maps $f \mapsto \partial^\beta f$ with $|\beta|<J.$ Given a bounded set $\mathcal{F}\subseteq \mathcal{B}_{\sigma}^{p,J}$, it only remains to establish conditions (ii) and (iii) of Theorem \ref{BanachMeanValueRK}.

We deal with condition (iii) first. We claim that for each $1 \leq j \leq N+M+1$, the collection of functions $\{D_j f\}_{f \in \mathcal{F}}$ is uniformly bounded on each compact set $F_m.$ Since $\mathcal{F}$ is bounded, we put $A:=\sup_{f \in \mathcal{F}} \|f\|_{\mathcal{B}^{p,J}_{\sigma}(D)}$. Condition \eqref{BergmanInequalityCondition} implies the uniform estimate
$$
\sup_{f \in \mathcal{F}} \sup_{w \in F_m}\left|\frac{\partial^\alpha f}{\partial z^\alpha}(w)\right| \leq C_{m,p,\sigma} A
$$
for all multi-indices $\alpha$ with $|\alpha|=J.$
 Note that we also trivially have that for any multi-index $\beta$ with $|\beta|<J,$ $\sup_{f \in \mathcal{F}}\left|\frac{\partial^{\beta} f}{\partial z^\beta}(z_0) \right| \leq A.$ This implies that for all such $\beta$, the collection of functions $\{\frac{\partial^{\beta} f}{\partial z^\beta}\}_{f \in \mathcal{F}}$ is uniformly bounded on $F_m.$   In particular, using an argument very similar to the one in Proposition \ref{BanachSpace}, we can show that when $\beta$ is any multi-index with $|\beta|=|\alpha|-1$, the following estimate holds uniformly for $f \in \mathcal{F}$ and $z \in F_m:$  
 $$
    \left| \frac{\partial{^\beta} f_j}{\partial z^\beta}(z)\right| \lesssim  \left| \frac{\partial^\beta f_j}{\partial z^\beta}(z_0)\right|+ \sup_{w \in K} \left( \sum_{|\alpha|=J} \left|\frac{\partial^\alpha f_j}{\partial z^\alpha}(w)\right|^2\right)^{1/2}|z-z_0| \leq C_{m,p,\sigma}A.
$$
This argument can be iterated to obtain the result for all multi-indices $\beta$ with $|\beta|<J.$ This establishes the uniform boundedness of the collection $\{D_j f\}_{f \in \mathcal{F}}$ on each compact $F_m.$ The argument also shows that $F_m$ can be replaced by an arbitrary compact set (with potentially a different constant). In particular, this establishes condition (iii) of Theorem \ref{BanachMeanValueRK}.

To establish condition (ii), we show that any collection of holomorphic functions on $D$, $\mathcal{G}$, that is uniformly bounded on compact sets is equicontinuous on each $F_m.$ Let $2\delta=\operatorname{dist}(F_m,\partial D)$. Then $F_m \subseteq \overline{D_{\delta}}$ and for any $z \in F_m$, the Euclidean ball $B(z, \delta)$ is contained in $\overline{D_{\delta}}.$ Let $M_1:= \sup_{f \in \mathcal{G}} \sup_{z \in \overline{D_{\delta}}}|f(z)|.$ First, if $f \in \mathcal{G}$ and $|z-w| \geq \frac{\delta}{6}$, we have
$$
    |f(z)-f(w)|  \leq 2M_1 = 2M_1\left(\frac{6}{\delta}\right)\left(\frac{\delta}{6}\right)  \leq M_2 |z-w|,
$$
where $M_2:= 2M_1(\frac{6}{\delta}).$
 So it suffices to prove the estimate for $z,w \in F_m$ and $|z-w|<\frac{\delta}{6}.$
 
In this case, notice $B(w,\frac{\delta}{3}) \subseteq B(z, \delta) \subseteq \overline{D_\delta}.$ The mean-value property give that
\begin{align*}
|f(z)-f(w)| & = \left| \frac{1}{V(B(z,\frac{\delta}{3}))}\int_{B(z,\frac{\delta}{3})}f\,dV-\frac{1}{V(B(w,\frac{\delta}{3}))}\int_{B(w,\frac{\delta}{3})}f\,dV\right|\\
&  = \left|\int_{B(z,\delta)}\left(\frac{f(\zeta) \chi_{B(z, \frac{\delta}{3})}(\zeta)}{V(B(z,\frac{\delta}{3}))}-  \frac{f(\zeta) \chi_{B(w, \frac{\delta}{3})}(\zeta)}{V(B(w,\frac{\delta}{3}))}\right) \, dV(\zeta)         \right|\\
& \leq M_1 \int_{B(z,\delta)}\left|\frac{\chi_{B(z, \frac{\delta}{3})}(\zeta)}{V(B(z,\frac{\delta}{3}))}-  \frac{ \chi_{B(w, \frac{\delta}{3})}(\zeta)}{V(B(w,\frac{\delta}{3}))}\right| \, dV(\zeta)\\
& \leq M_1 C_\delta\left(V\left(B\left(z,\frac{\delta}{3}\right) \setminus B\left(w,\frac{\delta}{3}\right) \right)+ V\left(B\left(w,\frac{\delta}{3}\right) \setminus B\left(z,\frac{\delta}{3}\right) \right)\right).
\end{align*}

To control the measures of the symmetric set differences above, note that $B(z,\frac{\delta}{3}-|z-w|) \subseteq B(w, \frac{\delta}{3}) $, and so
$$
   V\left(B\left(z,\frac{\delta}{3}\right) \setminus B\left(w,\frac{\delta}{3}\right) \right) \leq V\left(B\left(z,\frac{\delta}{3}\right) \setminus B\left(z,\frac{\delta}{3}-|z-w|\right) \right)
$$
$$
    = C_n\left( \left(\frac{\delta}{3}\right)^{2n}- \left(\frac{\delta}{3}-|z-w|\right)^{2n}\right) \leq C_{n, \delta} |z-w|.
$$
The other term is obviously handled similiarly.

\end{proof}



\begin{proof}[Proof of Corollary \ref{RadialBergmanSpaces}]
Note the condition $t>-1$ guarantees that $(-\rho)^t$ is integrable, see \cite{LS2012}*{Lemma 4.1}. We will apply Theorem \ref{BesovSobolevRK} to $\mathcal{B}_{\sigma}^{p,J}(D)$ with $J=0$ and $\sigma=(-\rho)^t.$

 We only need to verify that the weight  $\sigma=(-\rho)^t$ satisfies condition \eqref{BergmanInequalityCondition}. Let $K \subseteq D$ be compact. Note that there exist a radius $r_K$ and a compact set $K'$ with $K \subseteq K' \subseteq D$ such that for any $z \in K,$ the Euclidean ball $B(z,r_K) \subseteq K'.$ We then have, using the mean value property for holomorphic functions and H\"{o}lder's Inequality 
\begin{align*}
|f(z)| & \leq \frac{1}{ V(B(z,r_K))}\int_{B(z,r_K)} |f(w)| \mathop{dV(w)}\\
& \leq \frac{1}{\left(\inf_{\zeta \in K'} (-\rho(\zeta))^t \right) V(B(z,r_K))}\int_{B(z,r_K)} |f(w)| (-\rho(w))^t \mathop{dV(w)}\\
& \leq C_{K,t} \int_{D} |f(w)| (-\rho(w))^t \mathop{dV(w)}\\
& \leq C_{K,t} \left(\int_{D} (-\rho)^t \mathop{dV} \right)^{1/p'}  \left(\int_{D} |f|^p (-\rho)^t \mathop{d V}\right)^{1/p},
\end{align*}
so condition \eqref{BergmanInequalityCondition} is satisfied. 
\end{proof}




\begin{proof}[Proof of Corollary \ref{BPBergmanSpaces}]
We will apply Theorem \ref{BesovSobolevRK} to $\mathcal{B}_{\sigma}^{p,J}(D)$ with $J=0$ and $\sigma \in B_p.$ It is clear that $\sigma$ is a finite measure. We only must verify that $\sigma$ satisfies condition \eqref{BergmanInequalityCondition}. Let $K$ and $r_K$ be defined as in the proof of Corollary \ref{RadialBergmanSpaces}. Then we have, for $z \in K,$ using H\"{o}lder's inequality and the fact that $\sigma^{-1/(p-1)}$ is integrable on $D$:
\begin{align*}
|f(z)| & \leq \frac{1}{ V(B(z,r_K))}\int_{B(z,r_K)} |f(w)| \mathop{dV(w)}\\
& \leq C_K \int_{D} |f(w)| \mathop{dV(w)}\\
& \leq C_{K} \left(\int_{D} |f|^p \sigma  \mathop{dV}\right)^{1/p} \left(\int_{D} \sigma^{-1/(p-1)} \mathop{dV}\right)^{1/p'} \\
& \leq C_{K,p,\sigma} \left(\int_{D} |f|^p \sigma \mathop{d V}\right)^{1/p},
\end{align*}
so condition \eqref{BergmanInequalityCondition} is satisfied.
\end{proof}

\begin{proof}[Proof of Corollary \ref{HardySpaces}]
This is a straightforward application of Theorem \ref{BesovSobolevRK}.
\end{proof}

\begin{proof}[Proof of Corollary \ref{DirichletSpaces}]
This is a straightforward application of Theorem \ref{BesovSobolevRK}.
\end{proof}

\section{Applications of compactness criteria}\label{BergmanSection}


\subsection{Compactness of Toeplitz operators on the Bergman space} 
As before, let $\mathbb{B}_n$ denote the unit ball in $\mathbb{C}^n$ and let $\mathcal{A}^2(\mathbb{B}_n)$ denote the usual Bergman space on the unit ball, which is a reproducing kernel Hilbert space. The reproducing kernels and normalized reproducing kernels 
are respectively given by 
$$
    K_w(z)=\frac{1}{(1- z\overline{w})^{n+1}}\quad \text{and}\quad
    k_w(z)=\frac{(1-|w|^2)^{\frac{n+1}{2}}}{(1- z\overline{w})^{n+1}}.
$$
Recall that the ``$p$-normalized" reproducing kernels given by 
$$
    k_w^{(p)}(z):=\frac{K(z,w)}{\|K_w\|_{\mathcal{A}^2(\mathbb{B}_n)}^{2/p'}}= \frac{(1-|w|^2)^{\frac{n+1}{p'}}}{(1-z\overline{w})^{n+1}},
$$
and that the hyperbolic measure on $\mathbb{B}_n$ is given by 
$$
    d\lambda(z):=\frac{dv(z)}{(1-|z|^2)^{n+1}}.
$$
Let $\varphi_z$ denote the M\"obius transformation on $\mathbb{B}_n$ that interchanges $z$ and $0$. The 
Bergman metric on $\mathbb{B}_n$, $\beta$, is given by 
$$
    \beta(z,w):=\frac{1}{2}\log\left(\frac{1+|\varphi_z(w)|}{1-|\varphi_z(w)|}\right).
$$

For $z \in \mathbb{B}_n$ and $r>0$, set $D(z,r):=\{w\in \mathbb{B}_n: \beta(z,w)<r\}$. It is well-known that for any $r>0$, there exists $C_r>0$ such that 
\begin{align}\label{ComparableBergmannorms}
C_r^{-1}\leq \frac{\|K_z\|_{\mathcal{A}^2(\mathbb{B}_n)}}{\|K_w\|_{\mathcal{A}^2(\mathbb{B}_n)}}\leq C_r
\end{align}
for all $z,w \in \mathbb{B}_n$ with $\beta(z,w)<r$, see \cite{Zhu}*{Lemma 2.20}. It is also well-known that
$$
    \lambda(D(z,r))=\lambda(D(w,r))
$$
for all $z,w \in \mathbb{B}_n$ and $r>0$, see \cite{Zhu}*{Lemma 1.24}.


The following corollary is immediate from Theorem \ref{BanachFrameRK1} or Theorem \ref{BesovSobolevRK} (see Remark \ref{BergmanRemark}).
\begin{cor}\label{Bergmancompactnessdef}
Let $p \in (1,\infty)$ and $T$ be a bounded operator on $\mathcal{A}^p(\mathbb{B}_n)$. Then $T$ is compact on $\mathcal{A}^p(\mathbb{B}_n)$ if and only if 
$$
    \lim_{R\rightarrow \infty}\sup_{\substack{f \in \mathcal A^p(\mathbb{B}_n)\\ \|f\|_{\mathcal A^p(\mathbb{B}_n)\leq 1}}}\int_{\mathbb{B}_n\setminus D(0,R)}|\langle Tf,k_w^{(p')}\rangle_{\mathcal A^2(\mathbb{B}_n)}|^p\,d\lambda(w)=0.
$$
\end{cor}

\begin{prop}\label{Bergmancompactnesstest}
Let $p \in (1,\infty)$ and $T$ be a bounded operator on $\mathcal{A}^p(\mathbb{B}_n)$ such that 
\begin{align}\label{Bergmanboundedrows}
    \sup_{z \in \mathbb{B}_n}\int_{\mathbb{B}_n}|\langle T^*k_z,k_w\rangle_{\mathcal{A}^2(\mathbb{B}_n)}|\frac{\|K_z\|_{\mathcal{A}^2(\mathbb{B}_n)}^{1-\frac{2\delta}{p'(n+1)}}}{\|K_w\|_{\mathcal{A}^2(\mathbb{B}_n)}^{1-\frac{2\delta}{p'(n+1)}}}\,d\lambda(w)<\infty
\end{align}
for some $\delta \in \left(0,\min(p,p')\right)$. If
\begin{align}\label{Bergmandecayingcolumns}
\lim_{R\rightarrow\infty}\sup_{z \in \mathbb{B}_n}\int_{\mathbb{B}_n\setminus D(0,R)}|\langle Tk_z,k_w\rangle_{\mathcal{A}^2(\mathbb{B}_n)}|\frac{\|K_z\|_{\mathcal{A}^2(\mathbb{B}_n)}^{1-\frac{2\delta}{p(n+1)}}}{\|K_w\|_{\mathcal{A}^2(\mathbb{B}_n)}^{1-\frac{2\delta}{p(n+1)}}}\,d\lambda(w)=0,
\end{align}
then $T$ is compact on $\mathcal{A}^p(\mathbb{B}_n)$. 
\end{prop}

\begin{proof}
We verify the condition of Corollary \ref{Bergmancompactnessdef}. Let $f \in \mathcal{A}^p(\mathbb{B}_n)$ with $\|f\|_{\mathcal{A}^p(\mathbb{B}_n)}\leq 1$. Then, by Fubini's theorem
\begin{align*}
    \langle Tf, k_w^{(p')}\rangle_{\mathcal{A}^2(\mathbb{B}_n)} &=\left\langle T\left(\int_{\mathbb{B}_n} \langle f,k_z\rangle_{\mathcal{A}^2(\mathbb{B}_n)} k_z\,d\lambda(z)\right), k_w^{(p')}\right\rangle_{\mathcal{A}^2(\mathbb{B}_n)}\\
    &= \int_{\mathbb{B}_n} \langle f,k_z^{(p')}\rangle_{\mathcal{A}^2(\mathbb{B}_n)}\langle Tk_z^{(p)},k_w^{(p')}\rangle_{\mathcal{A}^2(\mathbb{B}_n)}\,d\lambda(z).
\end{align*}
By H\"older's inequality and the assumption \eqref{Bergmanboundedrows}, we have
\begin{align*}
    |\langle Tf,&k_w^{(p')}\rangle_{\mathcal{A}^2(\mathbb{B}_n)}|^p\leq \left(\int_{\mathbb{B}_n}|\langle f,k_z^{(p')}\rangle_{\mathcal{A}^2(\mathbb{B}_n)}||\langle Tk_z^{(p)},k_w^{(p')}\rangle_{\mathcal{A}^2(\mathbb{B}_n)}|\,d\lambda(z)\right)^p\\
    &\leq  \left(\int_{\mathbb{B}_n} |\langle Tk_z^{(p)},k_w^{(p')}\rangle_{\mathcal{A}^2(\mathbb{B}_n)}|\frac{\|K_w\|_{\mathcal{A}^2(\mathbb{B}_n)}^{\frac{2}{p}-\frac{2\delta}{p'(n+1)}}}{\|K_z\|_{\mathcal{A}^2(\mathbb{B}_n)}^{2-\frac{2}{p'}-\frac{2\delta}{p'(n+1)}}}\,d\lambda(z)\right)^{\frac{p}{p'}}\\
    &\quad\quad\times\left(\int_{\mathbb{B}_n}|\langle Tk_z^{(p)},k_w^{(p')}\rangle_{\mathcal{A}^2(\mathbb{B}_n)}||\langle f,k_z^{(p')}\rangle_{\mathcal{A}^2(\mathbb{B}_n)}|^p\left(\frac{\|K_z\|_{\mathcal{A}^2(\mathbb{B}_n)}^{\frac{2}{p}-\frac{2\delta}{p'(n+1)}}}{\|K_w\|_{\mathcal{A}^2(\mathbb{B}_n)}^{2-\frac{2}{p'}-\frac{2\delta}{p'(n+1)}}}\right)^{\frac{p}{p'}}\,d\lambda(z)\right)\\
    &\leq C^{\frac{p}{p'}} \int_{\mathbb{B}_n}|\langle Tk_z^{(p)},k_w^{(p')}\rangle_{\mathcal{A}^2(\mathbb{B}_n)}||\langle f,k_z^{(p')}\rangle_{\mathcal{A}^2(\mathbb{B}_n)}|^p\left(\frac{\|K_z\|_{\mathcal{A}^2(\mathbb{B}_n)}^{\frac{2}{p}-\frac{2\delta}{p'(n+1)}}}{\|K_w\|_{\mathcal{A}^2(\mathbb{B}_n)}^{2-\frac{2}{p'}-\frac{2\delta}{p'(n+1)}}}\right)^{\frac{p}{p'}}\,d\lambda(z),
\end{align*}
where $C$ is the finite constant from \eqref{Bergmanboundedrows}. Therefore, for any $R>0$, we have that \\$\int_{\mathbb{B}_n\setminus D(0,R)} |\langle Tf,k_w^{(p')}\rangle_{\mathcal{A}^2(\mathbb{B}_n)}|^p\,d\lambda(w)$ can be controlled above by
\begin{align*}
    &C^{\frac{p}{p'}}\int_{\mathbb{B}_n\setminus D(0,R)}\int_{\mathbb{B}_n}|\langle Tk_z^{(p)},k_w^{(p')}\rangle_{\mathcal{A}^2(\mathbb{B}_n)}||\langle f,k_z^{(p')}\rangle_{\mathcal{A}^2(\mathbb{B}_n)}|^p\left(\frac{\|K_z\|_{\mathcal{A}^2(\mathbb{B}_n)}^{\frac{2}{p}-\frac{2\delta}{p'(n+1)}}}{\|K_w\|_{\mathcal{A}^2(\mathbb{B}_n)}^{2-\frac{2}{p'}-\frac{2\delta}{p'(n+1)}}}\right)^{\frac{p}{p'}}\,d\lambda(z)d\lambda(w)\\
    &\quad\quad=C^{\frac{p}{p'}} \int_{\mathbb{B}_n}|\langle f,k_z^{(p')}\rangle_{\mathcal{A}^2(\mathbb{B}_n)}|^p\int_{\mathbb{B}_n\setminus D(0,R)} |\langle Tk_z,k_w\rangle_{\mathcal{A}^2(\mathbb{B}_n)}|\frac{\|K_z\|_{\mathcal{A}^2(\mathbb{B}_n)}^{1-\frac{2\delta}{p(n+1)}}}{\|K_w\|_{\mathcal{A}^2(\mathbb{B}_n)}^{1-\frac{2\delta}{p(n+1)}}}\,d\lambda(w)d\lambda(z).
\end{align*}

Let $\varepsilon>0$ be given. Apply the assumption \eqref{Bergmandecayingcolumns} to get a constant $R>0$ such that $\int_{\mathbb{B}_n\setminus D(0,R)} |\langle Tk_z,k_w\rangle_{\mathcal{A}^2(\mathbb{B}_n)}|\frac{\|K_z\|_{\mathcal{A}^2(\mathbb{B}_n)}^{1-\frac{2\delta}{p(n+1)}}}{\|K_w\|_{\mathcal{A}^2(\mathbb{B}_n)}^{1-\frac{2\delta}{p(n+1)}}}\,d\lambda(w) < \frac{\varepsilon}{C^{\frac{p}{p'}}}$ for all $z\in \mathbb{B}_n$. Then 
$$
    \int_{\mathbb{B}_n\setminus D(0,R)}|Tf|^p\,dv=\int_{\mathbb{B}_n\setminus D(0,R)}|\langle Tf,k_w^{(p')}\rangle_{\mathcal{A}^2(\mathbb{B}_n)}|^p\,d\lambda(w) < C^{\frac{p}{p'}}\frac{\varepsilon}{C^{\frac{p}{p'}}}\int_{\mathbb{B}_n}|f|^p\,dv \leq \varepsilon.
$$
Therefore $T$ is compact by Corollary \ref{Bergmancompactnessdef}.
\end{proof}

We now characterize the compact operators within a class of bounded and localized operators on $\mathcal{A}^p(\mathbb{B}_n)$. 
\begin{thm}\label{Bergmancompactoperatorcharacterization}
Let $p \in (1,\infty)$ and $T$ be a bounded operator on $\mathcal{A}^p(\mathbb{B}_n)$ satisfying \eqref{Bergmanboundedrows} and 
\begin{align}\label{Bergmandecayingcomplements}
    \lim_{R\rightarrow\infty}\sup_{z\in \mathbb{B}_n}\int_{\mathbb{B}_n\setminus D(z,R)}|\langle Tk_z,k_w\rangle_{\mathcal{A}^2(\mathbb{B}_n)}|\frac{\|K_z\|_{\mathcal{A}^2(\mathbb{B}_n)}^{1-\frac{2\delta}{p(n+1)}}}{\|K_w\|_{\mathcal{A}^2(\mathbb{B}_n)}^{1-\frac{2\delta}{p(n+1)}}}\,d\lambda(w)=0.
\end{align}
Then $T$ is compact on $\mathcal{A}^p(\mathbb{B}_n)$ if and only if $\widetilde{T}(z)\rightarrow 0$ as $z\rightarrow 1^-$. 
\end{thm}

\begin{proof}
The forward direction is clear (in particular, use the well-known fact that the p-normalized kernels $k_z^{(p)}$ converge weakly to $0$ in $\mathcal{A}^p(\mathbb{B}_n)$ as $z \rightarrow 1^-$), so we only consider the reverse direction. We will verify condition \eqref{Bergmandecayingcolumns} and establish the theorem by applying Proposition \ref{Bergmancompactnesstest}. We will use the following condition which is implied by the vanishing Berezin transform hypothesis (see \cite{IMW2015}): for each $R>0$, we have 
\begin{align}\label{WeakBerezinCondition}
    \lim_{z\rightarrow 1^-}\sup_{w \in D(z,R)}|\langle Tk_z,k_w\rangle_{\mathcal{A}^2(\mathbb{B}_n)}|=0.
\end{align}

Let $\varepsilon >0$ be given. Apply assumption \eqref{Bergmandecayingcomplements} to find $R_0>0$ such that 
$$
    \int_{\mathbb{B}_n\setminus D(z,R_0)}|\langle Tk_z,k_w\rangle_{\mathcal{A}^2(\mathbb{B}_n)}|\frac{\|K_z\|_{\mathcal{A}^2(\mathbb{B}_n)}^{1-\frac{2\delta}{p(n+1)}}}{\|K_w\|_{\mathcal{A}^2(\mathbb{B}_n)}^{1-\frac{2\delta}{p(n+1)}}}\,d\lambda(w)<\frac{\varepsilon}{2}
$$
for all $z \in \mathbb{B}_n$. For such $R_0$, we next use condition \eqref{WeakBerezinCondition} to get $N_0>0$ such that \\$|\langle Tk_z,k_w\rangle_{\mathcal{A}^2(\mathbb{B}_n)}|<\frac{\varepsilon}{2 C_{R_0}^{1-\frac{2\delta}{p(n+1)}}\lambda(D(0,R_0))}$ for all $z \in \mathbb{B}_n\setminus D(0,N_0)$ and $w \in D(z,R_0)$, where $C_{R_0}$ is the constant given in \eqref{ComparableBergmannorms}. Set $R=N_0+R_0$. 

Write
\begin{align*}
    \int_{\mathbb{B}_n\setminus D(0,R)} &|\langle Tk_z,k_w\rangle_{\mathcal{A}^2(\mathbb{B}_n)}|\frac{\|K_z\|_{\mathcal{A}^2(\mathbb{B}_n)}^{1-\frac{2\delta}{p(n+1)}}}{\|K_w\|_{\mathcal{A}^2(\mathbb{B}_n)}^{1-\frac{2\delta}{p(n+1)}}}\,d\lambda(w)\\
    &= \int_{\mathbb{B}_n\setminus (D(0,R)\cup D(z,R_0))} |\langle Tk_z,k_w\rangle_{\mathcal{A}^2(\mathbb{B}_n)}|\frac{\|K_z\|_{\mathcal{A}^2(\mathbb{B}_n)}^{1-\frac{2\delta}{p(n+1)}}}{\|K_w\|_{\mathcal{A}^2(\mathbb{B}_n)}^{1-\frac{2\delta}{p(n+1)}}}\,d\lambda(w)\\
    &\quad\quad+\int_{(\mathbb{B}_n\setminus D(0,R))\cap D(z,R_0)} |\langle Tk_z,k_w\rangle_{\mathcal{A}^2(\mathbb{B}_n)}|\frac{\|K_z\|_{\mathcal{A}^2(\mathbb{B}_n)}^{1-\frac{2\delta}{p(n+1)}}}{\|K_w\|_{\mathcal{A}^2(\mathbb{B}_n)}^{1-\frac{2\delta}{p(n+1)}}}\,d\lambda(w).
\end{align*}
The first term above is controlled by the choice of $R_0$:
$$
    \int_{\mathbb{B}_n\setminus (D(0,R)\cup D(z,R_0))} |\langle Tk_z,k_w\rangle_{\mathcal{A}^2(\mathbb{B}_n)}|\frac{\|K_z\|_{\mathcal{A}^2(\mathbb{B}_n)}^{1-\frac{2\delta}{p(n+1)}}}{\|K_w\|_{\mathcal{A}^2(\mathbb{B}_n)}^{1-\frac{2\delta}{p(n+1)}}}\,d\lambda(w)<\frac{\varepsilon}{2}.
$$
For the second term, we only need to consider the case when $(\mathbb{B}_n \setminus D(0,R)) \cap D(z,R_0) \neq \emptyset$. If $w \in (\mathbb{B}_n \setminus D(0,R)) \cap D(z,R_0)$, then $d(z,0)\ge d(w,0)-d(z,w)\ge R- R_0 = N_0$. Thus we can control the second term by choice of $N_0$:
\begin{align*}
    \int_{(\mathbb{B}_n \setminus D(z_0,R)) \cap D(z,R_0)}& |\langle Tk_z,k_w\rangle_{\mathcal{A}^2(\mathbb{B}_n)}|\frac{\|K_z\|_{\mathcal{A}^2(\mathbb{B}_n)}^{1-\frac{2\delta}{p(n+1)}}}{\|K_w\|_{\mathcal{A}^2(\mathbb{B}_n)}^{1-\frac{2\delta}{p(n+1)}}}\,d\lambda(w)\\
    &< C_{R_0}^{1-\frac{2\delta}{p(n+1)}}\lambda(D(z,R_0))\frac{\varepsilon}{2C_{R_0}^{1-\frac{2\delta}{p(n+1)}}\lambda(D(0,R_0))}\\
    &\leq \frac{\varepsilon}{2}.
\end{align*}
Therefore 
$$
     \int_{\mathbb{B}_n\setminus D(0,R)}|\langle Tk_z,k_w\rangle|\frac{\|K_z\|_{\mathcal{A}^2(\mathbb{B}_n)}^{1-\frac{2\delta}{p(n+1)}}}{\|K_w\|_{\mathcal{A}^2(\mathbb{B}_n)}^{1-\frac{2\delta}{p(n+1)}}}\,d\lambda(w)<\frac{\varepsilon}{2}+\frac{\varepsilon}{2}=\varepsilon,
$$
completing the proof.
\end{proof}

\begin{proof}[Proof of Theorem \ref{ToeplitzCompactness}]
Conditions \eqref{Bergmanboundedrows} and \eqref{Bergmandecayingcomplements} are satisfied for $T$ by \cite{IMW2015}*{Proposition 2.2} (via the vanishing Rudin-Forelli estimates). We pass to the case where $T$ is a finite sum of finite products of Toeplitz operators using the argument in \cite{IMW2015}*{Proposition 2.3}.
\end{proof}

\subsection{Compactness of little Hankel Operators on the Hardy space}\label{HardySection}



We prove the sufficiency portion of Theorem \ref{HankelCompactness} using Corollary \ref{HardySpaces}. Our proof is different from the usual argument that uses the Stone-Weierstrass theorem and the fact that little Hankel operators with polynomial symbols are of finite rank. Our proof does not use the approximation of $H_g$ by finite rank operators.

\begin{proof}[Proof of Theorem \ref{HankelCompactness}]
Suppose $g \in \text{VMOA}.$ By \cite{Zhu}*{Theorem 5.18}, there exists a function $\tilde{g} \in C(\partial \mathbb{B}_n)$ such that $S(\tilde{g})=g.$ Since $S(g \overline{f})=H_{S(g)}f$ (as densely defined operators)  for any $g \in L^2(\mathbb{B}_n)$  and  $f \in \mathcal{H}^2(\partial \mathbb{B}_n)$, we may assume without loss of generality that $g \in C(\partial \mathbb{B}_n).$ 

 Let $1 \leq \ell \leq n$ be an index and note that for $z \in \mathbb{B}_n,$ 
$$
    \dfrac{\partial}{\partial z_{\ell}}H_g(f)(z)= \int_{\partial \mathbb{B}_n} \frac{\overline{w_{\ell}} g(w) \overline{f(w)}}{(1-z\overline{w})^{n+1}} \mathop{d s(w)}.
$$ 
Using our Corollary \ref{HardySpaces}, it suffices to show that given $\varepsilon>0$, there exists $R$ sufficiently close to $1$ such that
$$
    \sup_{\substack{f \in \mathcal{H}^2(\mathbb{B}_n) \\ \|f\|_{\mathcal{H}^2(\mathbb{B}_n)} \leq 1}} \sum_{\ell=1}^{n} \int_{(R \mathbb{B}_n)^c} (1-|z|^2) \left|\dfrac{\partial}{\partial z_{\ell}} S(g \overline{f})(z) \right| ^2 \mathop{dv(z)}<\varepsilon. 
$$ 

Clearly, it suffices to prove the bound for a fixed index $\ell$ (just replace $\varepsilon$ by $\frac{\varepsilon}{n}$). Take any $f$ in $\mathcal{H}^2(\mathbb{B}_n)$ with $\|f\|_{\mathcal{H}^2(\mathbb{B}_n)} \leq 1.$ Consider a decomposition of $\mathbb{C}^n=\mathbb{R}^{2n}$ into closed cubes with disjoint interiors $Q_j$ with side length $r.$ Let $D_j=Q_j \cap \partial \mathbb{B}_n$. Clearly, these sets are non-empty for only finitely many $j$, and we may assume that $\partial \mathbb{B}_n= \bigcup_{j=1}^{N_r} D_j,$ where the interiors of $D_j$ are pairwise disjoint and $N_r$ is a positive integer that depends on the side length $r$. Also note that the sets $D_j$ have the following finite overlap property: there exists a constant $K$ (independent of $r$) such that each set $D_k$ intersects at most $K$ members of $\{D_j\}_{j=1}^{N_r}$. 
Now, use the (uniform) continuity of $g$ to choose $r$ such that if $z,w \in D_j,$ then $|g(z)-g(w)|< \sqrt{\frac{\varepsilon}{8 C_{\mathcal{H}^2}K N_r^2}},$ where $C_{\mathcal{H}^2}$ is a constant such that
$$
    |f(0)|^2+ \sum_{\ell=1}^n\int_{\mathbb{B}_n} \left|\frac{\partial f}{\partial z_\ell}(z)\right|^2 (1-|z|^2) \mathop{d v(z)} \leq C_{\mathcal{H}^2} \int_{\partial \mathbb{B}_n} |f(\zeta)|^2 \, ds(\zeta)
$$
for all $f \in \mathcal{H}^2(\mathbb{B}_n).$

Define $C_j=\{z \in \mathbb{B}_n: \pi(z) \in D_j \text{ or } z=0\}$, where $\pi(z)=z/|z|$ is the standard radial projection to the boundary. It is straightforward to verify there exists a constant $M_r$ that depends on $r$ such that if $w \in D_j$ and $z \in C_k$ with $D_j \cap D_k= \emptyset,$ then $\frac{1}{|1-z\overline{w}|^{n+1}}<M_r.$ Choose a sample point $z_j \in D_j$ for each $j$. Define the functions $g_j:=g-g(z_j)$ for $1 \leq j \leq N_r$ and use the mean value property to write
\begin{align*}
&\int_{(R \mathbb{B}_n)^c} (1-|z|^2) \left|\int_{\partial \mathbb{B}_n} \frac{\overline{w_\ell} g(w) \overline{f(w)}}{(1-z\overline{w} )^{n+1}} \mathop{ds(w)} \right| ^2 \mathop{dv(z)}\\
& = \sum_{j=1}^{N_r}\int_{(R \mathbb{B}_n)^c\cap C_j} (1-|z|^2) \left|\int_{\partial \mathbb{B}_n} \frac{\overline{w_\ell} g_j(w)\overline{f(w)}}{(1-z\overline{w})^{n+1}} \mathop{ds(w)} \right| ^2 \mathop{dv(z)}\\
 & \leq N_r^2 \sum_{j,k=1}^{N_r}\int_{(R \mathbb{B}_n)^c \cap C_j} (1-|z|^2) \left|\int_{D_k} \frac{\overline{w_\ell} g_j(w)\overline{f(w)}}{(1-z\overline{w})^{n+1}} \mathop{ds(w)} \right| ^2 \mathop{dv(z)}
 \end{align*}
\begin{align*}
& = N_r^2 \sum_{\substack{1 \leq j,k \leq N_r \\  D_j \cap D_k= \emptyset}}\int_{(R \mathbb{B}_n)^c \cap C_j} (1-|z|^2) \left|\int_{D_k} \frac{\overline{w_\ell} g_j(w)\overline{f(w)}}{(1-z\overline{w})^{n+1}} \mathop{ds(w)} \right| ^2 \mathop{dv(z)}\\
& \quad\quad+ N_r^2 \sum_{\substack{1 \leq j,k \leq N_r \\ D_j \cap D_k \neq  \emptyset}}\int_{(R \mathbb{B}_n)^c \cap C_j} (1-|z|^2) \left|\int_{D_k} \frac{\overline{w_\ell} g_j(w)\overline{f(w)}}{(1-z\overline{w})^{n+1}} \mathop{ds(w)} \right| ^2 \mathop{dv(z)}\\
& =: N_r^2[(\text{I}) + (\text{II})].
\end{align*}

We control $(\text{I})$ using the bound on the kernel, the disjointness of the $D_k$, and H\"{o}lder's inequality together with the fact that $\|f\|_{\mathcal{H}^2(\mathbb{B}_n)}$ as follows:
\begin{align*}
(\text{I}) & \leq \sum_{\substack{1 \leq j,k \leq N_r \\  D_j \cap D_k= \emptyset}}\int_{(R \mathbb{B}_n)^c \cap C_j} (1-|z|^2) \left(\int_{D_k} \frac{ |g_j(w)||f(w)|}{|1-z\overline{w}|^{n+1}} \mathop{ds(w)} \right) ^2 \mathop{dv(z)}\\
& \leq  M_r^2 \sum_{j=1}^{N_r}\int_{(R \mathbb{B}_n)^c \cap C_j} (1-|z|^2) \left(\int_{\partial \mathbb{B}_n}  |g_j(w)||f(w)| \mathop{ds(w)} \right) ^2 \mathop{dv(z)}\\
& \leq  M_r^2 \sum_{j=1}^{N_r} \|g_j\|_{\mathcal{H}^2(\mathbb{B}_n)}^2  \int_{(R \mathbb{B}_n)^c \cap C_j}  (1-|z|^2)\mathop{dv(z)}\\
& \leq 4   M_{r}^2 \|g\|_{L^\infty(\partial\mathbb{B}_n)}^2 \int_{(R \mathbb{B}_n)^c}  (1-|z|^2)\mathop{dv(z)}.
\end{align*}
Choosing $R$ sufficiently close to $1$ so that $\int_{(R \mathbb{B}_n)^c}  (1-|z|^2)\mathop{dv(z)} \leq \frac{\varepsilon}{8 N_r^2  M_r^2 \|g\|_{L^{\infty}(\partial\mathbb{B}_n)}^2}$, we deduce that $(\text{I})<\frac{\varepsilon}{2 N_r^2}.$

To control $(\text{II})$, we use the equivalence of $\mathcal{H}^2(\mathbb{B}_n)$ norms and the boundedness of the Szeg\H{o} projection on $L^2(\partial \mathbb{B}_n).$ If $z \in D_k$ and $D_k \cap D_j \neq \emptyset,$ then the continuity of $g$ together with the triangle inequality implies that $|g_j(z)|=|g(z)-g(z_j)|  \leq 2 \sqrt{\frac{\varepsilon}{8 C_{\mathcal{H}^2}K N_r^2}}$. Thus 
\begin{align*}
(\text{II}) & \leq \sum_{\substack{1 \leq j,k \leq N_r \\  D_j \cap D_k \neq  \emptyset}}\int_{\mathbb{B}_n} (1-|z|^2) \left|\dfrac{\partial}{\partial z_\ell} S(\chi_{D_k} g_j \overline{f})(z) \right| ^2 \mathop{dv(z)}\\
& \leq C_{\mathcal{H}^2} \sum_{\substack{1 \leq j,k \leq N_r \\  D_j \cap D_k \neq  \emptyset}}\int_{\partial \mathbb{B}_n} |S(\chi_{D_k}g_j \overline{f})(z)|^2 \mathop{ds(z)}\\
& \leq C_{\mathcal{H}^2} \sum_{\substack{1 \leq j,k \leq N_r \\  D_j \cap D_k \neq  \emptyset}} \int_{D_k} |g_j(z)f(z)|^2 \mathop{d s  (z)}\\
& <\frac{C_{\mathcal{H}^2} K \varepsilon}{2 C_{\mathcal{H}^2} K N_r^2 } \sum_{j=1}^{N_r} \int_{D_j} |f(z)|^2 \mathop{d s (z)}\\
& \leq \frac{\varepsilon}{2 N_r^2}.
\end{align*}

Putting all this together, we deduce that for this choice of $R,$
$$\int_{(R \mathbb{B}_n)^c} (1-|z|^2) \left|\int_{\partial \mathbb{B}_n} \frac{\overline{w_\ell} g(w) \overline{f(w)}}{(1-z\overline{w})^{n+1}} \mathop{ds(w)} \right| ^2 \mathop{dv(z)} < \frac{\varepsilon}{2}+\frac{\varepsilon}{2}=\varepsilon,$$
which completes the proof.

\end{proof}

\section{Acknowledgments}
M. Mitkovski's research is supported in part by National Science Foundation grant DMS \#2000236.

N. A. Wagner's research is supported in part by National Science Foundation grant DGE \#1745038.

B. D. Wick's research is supported in part by National Science Foundation grants DMS \#1800057, \#2054863, \#20000510 and Australian Research Council - DP 220100285.

\end{document}